\documentclass[12pt]{amsart}

\usepackage{graphicx}
\usepackage{eucal}
\usepackage{xcolor}
\usepackage{subcaption}
\usepackage{amsfonts,amssymb,amscd}

\newtheorem{theorem}{Theorem}[section]

\theoremstyle{definition}
\newtheorem{definition}[theorem]{Definition}
\newtheorem{example}[theorem]{Example}

\newtheorem{corollary}[theorem]{Corollary}

\theoremstyle{remark}
\newtheorem{remark}[theorem]{Remark}
\numberwithin{equation}{section}
\usepackage{graphicx}
\usepackage{hyperref}
\usepackage{geometry}
\usepackage[all]{xy}
\begin{document}
\title{\sc Census of bounded curvature paths}
\author{Jean D\'iaz}
\author{Jos\'{e} Ayala}
\address{FIA, Universidad Arturo Prat, Iquique, Chile}
%\email{jayalhoff@gmail.com}
\email{jayalhoff@gmail.com - jeanpdp@gmail.com}
\subjclass[2000]{}
\date{}
\dedicatory{}
\maketitle

\begin{abstract} A bounded curvature path is a continuously differentiable piece-wise C2 path with bounded absolute curvature connecting two points in the tangent bundle of a surface. These paths have been widely considered in computer science and engineering since the bound on curvature models the trajectory of the motion of robots under turning circle constraints. Analyzing global properties of spaces of bounded curvature paths is not a simple matter since the length variation between length minimizers of arbitrary close endpoints or directions is in many cases discontinuous. In this note, we develop a simple technology allowing us to partition the space of spaces of bounded curvature paths into one-parameter families. These families of spaces are classified in terms of the type of connected components their elements have (homotopy classes, isotopy classes, or isolated points) as we vary a parameter defined in the reals. Consequently, we answer a question raised by Dubins (Pac J Math 11(2):471–481, 1961).
\end{abstract}

%----------------------------------------------------------------------------------------
%	PRELUDE
%----------------------------------------------------------------------------------------
\section{Prelude}

It is well known that any two plane curves both closed or with different endpoints are homotopic.  Graustein, and Whitney in 1937, independently proved that not any two planar closed curves are regularly homotopic (homotopic through immersions) \cite{whitney}. Markov in 1857 considered several optimization problems relating a bound on curvature with the design of railroads \cite{markov}. But, it was only in 1957 that bounded curvature paths were rigorously introduced by Dubins when bounded curvature paths of minimal length were first characterized \cite{dubins 1}. 

Fix two elements in the tangent bundle of the Euclidean plane $(x,X),(y,Y)\in T{\mathbb R}^2$. Informally, a planar bounded curvature path is a $C^1$ and piecewise $C^2$ path starting at $x$, finishing at $y$; with tangent vectors at these points $X$ and $Y$ respectively, having absolute curvature bounded by $\kappa=\frac{1}{r}>0$. Here $r$ is the minimum allowed radius of curvature. The piecewise $C^2$ property comes naturally due to the nature of the length minimizers \cite{dubins 1}\footnote{Dubins proved that bounded curvature paths of minimal length are concatenations of two arcs of a circle with a line segment in between, or three arcs of a circle, or any subset of these. The so-called {\sc csc}-{\sc ccc} paths.}.

In 1961 Dubins raised fundamental questions about the topology of the spaces of bounded curvature paths \cite{dubins 2}. ``Here we only begin the exploration, raise some questions that we hope will prove stimulating, and invite others to discover the proofs of the definite theorems, proofs that have eluded us'' see pp. 471 in \cite{dubins 2}. Fifty years later the fundamental questions proposed by Dubins were answered through the papers \cite{paperb, papera, paperc, paperd}.  In addition, the classification of the homotopy classes of curves with bounded absolute curvature, having fixed initial and final positions, and variable initial and final directions was achieved in \cite{papere}. 

In this note, we develop an elementary framework enabling us to parametrize families of spaces of bounded curvature paths. These families share similar types of connected components, being these: isolated points, homotopy, or isotopy classes, see Theorem \ref{maincensus1} and Definition \ref{def:spaces}. In particular, we answer a question raised by Dubins in 1961 \cite{dubins 2} by explicitly describing the set of endpoints $(x,X),(y,Y)\in T{\mathbb R}^2$ so that the space of bounded curvature paths starting at $x$, finishing at $y$; with tangent vectors at these points $X$ and $Y$ respectively admits a bounded isotopy class, see Theorem \ref{noopnoclofib} and Corollary \ref{noopnoclo}. We conclude by presenting an updated (parametric) version of the classification theorem for homotopy classes of bounded curvature paths in \cite{paperd}, by incorporating the results here obtained, see Theorem \ref{paramclass}. Our results can be extended without much effort for paths in the hyperbolic 2-space. 
 
 This article is the culmination of a program devoted to classify the homotopy classes of bounded curvature paths \cite{paperc, paperd}, and the minimal length elements in homotopy classes \cite{paperb, papera}. We recommend the reader from time to time refer to our previous work \cite{paperb, papere, papera, paperc, paperd}.  We conclude by presenting an Appendix that can be read independently. This Appendix considers further examples and questions about computational aspects of connected components, and deformations of piecewise constant bounded curvature paths.

 There is a vast literature on bounded curvature paths from the theoretical computer science point of view. We encourage the reader to refer to \cite{aga1, baker, buasanei2, bui, fortune, jacobs, lavalle, reif, rus}.  Bounded curvature paths have been applied to many real-life problems since a bound on curvature models the trajectory of the motion of wheeled vehicles, and drones also called unmanned aerial vehicles (UAV). We mention only \cite{brazil 1, chang, duindan, ny, owen1, soures, tso1}. Literature on the topology and geometry of spaces of bounded curvature paths can be found in \cite{{paperb}, {papere}, {papera}, {paperc}, {paperd}, {dubins 1}, {dubins 2}, {reeds}, {saldanha}, {sus}}. 

The illustrations here presented have been imported from Dubins Explorer, a software for bounded curvature paths \cite{dubinsexplorer}. 

%----------------------------------------------------------------------------------------
%	ON SPACES OF BOUNDED CURVATURE PATHS
%----------------------------------------------------------------------------------------
\section{On spaces of bounded curvature paths}

For the convenience of the reader, we include relevant material from our previous work in \cite{paperb, papere, papera, paperc, paperd}. Denote by $T{\mathbb R}^2$ the tangent bundle of ${\mathbb R}^2$. Recall that the elements in $T{\mathbb R}^2$ are pairs $(x,X)$ denoted here for short by {\sc x}. The first coordinate of such a pair corresponds to a point in ${\mathbb R}^2$ and the second to a tangent vector to ${\mathbb R}^2$ at $x$.

\begin{definition} \label{defbcp} Given $(x,X),(y,Y) \in T{\mathbb R}^2$,  a path $\gamma: [0,s]\rightarrow {\mathbb R}^2$ connecting these points is a {\it bounded curvature path} if:
\end{definition}
 \begin{itemize}
\item $\gamma$ is $C^1$ and piecewise $C^2$;
\item $\gamma$ is parametrized by arc length (i.e $||\gamma'(t)||=1$ for all $t\in [0,s]$);
\item $\gamma(0)=x$,  $\gamma'(0)=X$;  $\gamma(s)=y$,  $\gamma'(s)=Y$;
\item $||\gamma''(t)||\leq \kappa$, for all $t\in [0,s]$ when defined, $\kappa>0$ a constant.
\end{itemize}

The first item means that a bounded curvature path has continuous first derivative and piecewise continuous second derivative. Minimal length elements in spaces of paths satisfying the last three items in Definition \ref{defbcp} are in fact $C^1$ and piecewise $C^2$. For the third item, without loss of generality, we extend the domain of $\gamma$ to $(-\epsilon,s+\epsilon)$ for $\epsilon>0$. Sometimes we describe the third item as the endpoint condition. The fourth item means that bounded curvature paths have absolute curvature bounded above by a positive constant. Without loss of generality, we consider $\kappa=1$. 

The unit tangent bundle $UT\mathbb R^2$ is equipped with a natural projection $p : UT{\mathbb R}^2 \rightarrow {\mathbb R}^2$. Note that $p^{-1}(y)$ is $\mathbb S^1$ for all $y \in{\mathbb R}^2$. The space of endpoints is a circle bundle over ${\mathbb R}^2$.

\begin{remark}\label{coo}{\it (Coordinate system and angle orientation).} 
\begin{itemize}
\item For the given $(x,X), (y,Y) \in T{\mathbb R}^2$ in Definition \ref{defbcp} we consider a coordinate system so that the origin is identified with $x$, and $X$ with the first canonical vector in the standard basis $\{X=e_1,e_2 \}$ for $\mathbb R^2$. 
\item When measuring angles we consider the positive orientation to be traveled counterclockwise. 
\end{itemize}
\end{remark}

Dubins \cite{dubins 1} proved that the length minimizer bounded curvature paths are necessarily a concatenation of an arc of a unit radius circle, followed by a line segment, followed by an arc of a unit radius circle, the so-called {\sc csc} paths. Or, a concatenation of three arcs of unit radius circles, the so-called {\sc ccc} paths. After considering {\sc r} be a circle traveled to the right and {\sc l} a circle travelled to the left we obtain six possible types of paths, namely {\sc lsl}, {\sc rsr}, {\sc lsr}, {\sc rsl}, {\sc lrl} and {\sc rlr}. Paths having one of these types are here called Dubins paths. Note that we are considering Dubins path to be local not necessarily global minimum of length.  

In the following paragraphs, we illustrate through examples the richness of the theory of bounded curvature paths. Its features come from the constraints these curves satisfy. These constraints lead to interesting interactions between metric geometry and computational mathematics. 

\begin{example} \label{ex:1}\hfill 
\begin{enumerate}
\item Recall that length minimizers are considered for establishing distance between points in a manifold. This approach is not suitable when considering bounded curvature paths, since in many cases, the length variation between length minimizers of arbitrarily close endpoints or directions is discontinuous. 

Consider $(x,X),(y,Y_\theta) \in T \mathbb R^2$, $\theta \in \mathbb R$, with $\kappa=1$:
\begin{itemize}
\item $x=(0,0)$; $X=e^{2\pi i}\in T_x\mathbb R^2$.
\item $y=(1,1)$; $\mbox{\it Y}_{\theta}=e^{\theta i}\in T_y\mathbb R^2$.
\end{itemize}

\noindent Discontinuities for the length of the length minimizers happen when perturbing around $\mbox{\it Y}_{\frac{\pi}{2}}$, see Fig.~\ref{figdiscde}. The sudden jumps in length suggest the existence of isolated points, see Theorem 3.9 in \cite{paperc}. In fact, the path in Fig. \ref{figdiscde} left is an isolated point in the space of bounded curvature paths from $(x,X)$ to $(y,Y_{\frac{\pi}{2}})$. 

The path in Fig.~\ref{figdiscde} right illustrates a discontinuity after perturbing the final location to $y'=(1-\epsilon,1-\epsilon)$ for $\epsilon>0$ small. Note that length minimisers may not be embedded paths.

%----------------------------------------------------------------------------------------
%	FIGURE 1
%----------------------------------------------------------------------------------------
\begin{figure}[h]
	\centering
	\includegraphics[width=1\textwidth,angle=0]{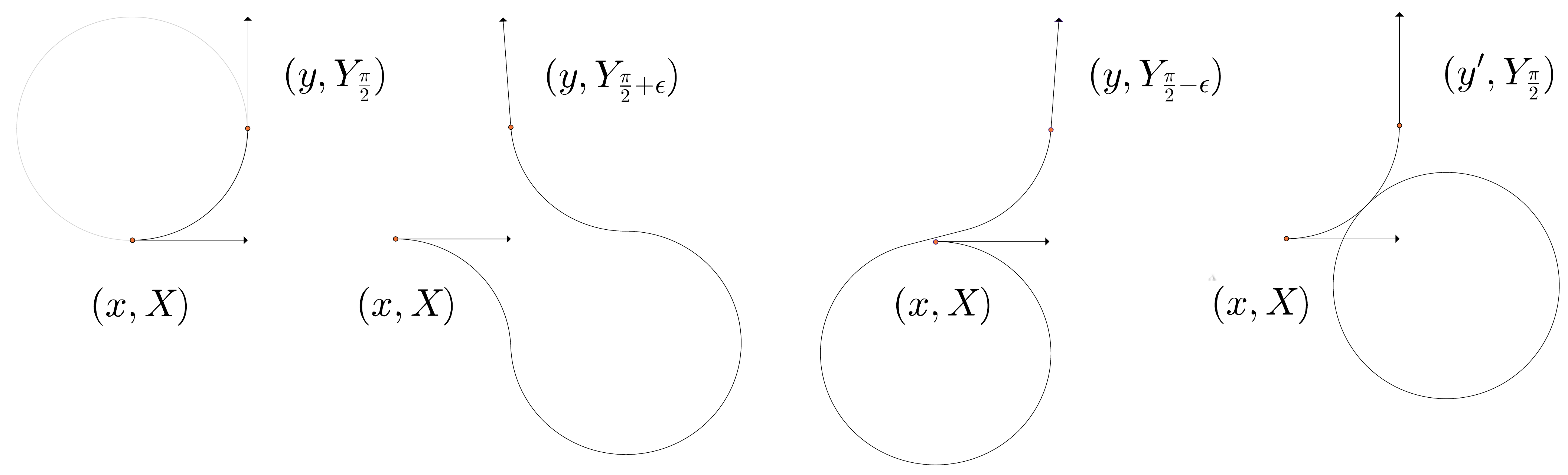}
	\caption{Examples of length minimizers in their respective path space. The three paths at the right are the result of small perturbations to the final position or direction of $(y,Y_{\frac{\pi}{2}})$, for $\epsilon>0$. After applying Dubins' characterization for the length minimizers \cite{dubins 1} a simple numerical experiment shows the existence of length discontinuities.}
	\label{figdiscde}
\end{figure}

\item Spaces of bounded curvature paths have several local minima of length, see Fig. \ref{sixdubb1}. 

Consider $\mbox{\sc x}=(x,X),\mbox{\sc y}=(y,Y)\in T\mathbb R^2$ with $\kappa=1$:
\begin{itemize}
\item  $x=(0,0)$; $X=e^{2\pi i}\in T_x\mathbb R^2$.
\item $y=(-2,1)$; $Y=e^{-\frac{\pi}{4} i}\in T_y\mathbb R^2$.
\end{itemize}

 By recursively applying the methods in \cite{papera} for obtaining the {\sc csc}-{\sc ccc} charactarization for the length minimizers \cite{paperb, papera, buasanei1, dubins 1, johnson} we obtain all the local minima of length. 
 
 By Proposition 4.4 in \cite{paperd} the paths $\gamma_0$ and $\gamma_5$ are homotopic without violating the curvature bound throughout the deformation. The same applies for $\gamma_1$ and $\gamma_4$. By Proposition 4.3 in \cite{paperd} the paths $\gamma_2$ and $\gamma_3$ lie in the same homotopy class of bounded curvature paths. 

%----------------------------------------------------------------------------------------
%	FIGURE 2
%----------------------------------------------------------------------------------------
\begin{figure}[h]
	\centering
	\includegraphics[width=.9\textwidth,angle=0]{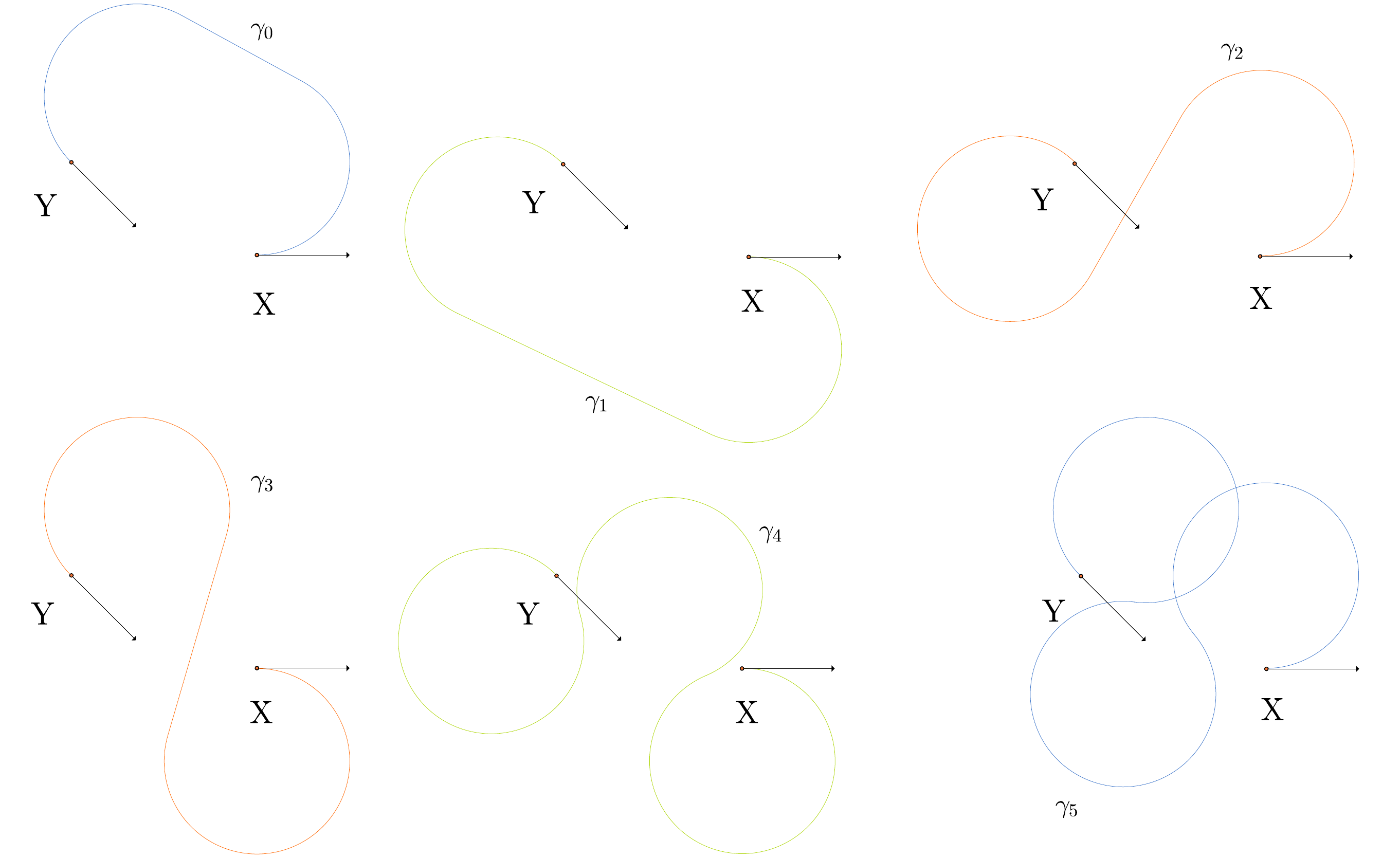}
	\caption{Spaces of bounded curvature paths have several local minima of length. Paths with matching colors are homotopic without violating the curvature bound throughout the deformation.}
	\label{sixdubb1}
\end{figure}

\item Another interesting feature is that the symmetry property metrics satisfy is in general violated. For example, the length minimizer from $(x,X)$ to $(y,\mbox{\it Y}_{\frac{\pi}{2}})$ has lenght $\frac{\pi }{4}$, see Fig.~\ref{figdiscde} left. On the other hand, the length minimizer from $(y,\mbox{\it Y}_{\frac{\pi}{2}})$ to $(x,X)$ has length $\frac{3\pi }{4}$, see Theorem 4.6 in \cite{paperb}. 

\item The classification of the homotopy classes of bounded curvature paths was obtained in \cite{paperd}. A crucial step was to prove that for certain $(x,X),(y,Y)\in T{\mathbb R}^2$ there exists a bounded region $\Omega\subset \mathbb R^2$ that ``traps'' embedded bounded curvature paths. That is, no embedded bounded curvature path whose image is in $\Omega$ can be deformed (while preserving the curvature bound throughout the deformation) to a path having a point not in $\Omega$, see Definition 4.1 in \cite{paperc}. In \cite{paperd} we proved that these ``trapped regions'' are the domain of elements in isotopy classes of bounded curvature paths, we refer to these as {\bf bounded isotopy classes}. Discontinuities may also occur in the formation of trapped regions. These ideas will be discussed in subsection \ref{construct}. 

Consider $(x,X),(y,Y)\in T\mathbb R^2$ with $\kappa=1$. For $\epsilon>0$ small, the spaces of bounded curvature paths satisfying:

\begin{itemize}
\item  $x=(0,0)$; $X=e^{2\pi i}\in T_x\mathbb R^2$
\item $y=(1+\epsilon,1+\epsilon)$; $Y=e^{\frac{\pi}{2}i}\in T_y\mathbb R^2$
\end{itemize}

\noindent have associated a region that ``traps'' embedded bounded curvature paths, see Fig. \ref{motiv} right. More generally, spaces satisfying the previous two conditions admit a bounded isotopy class of bounded curvature paths, see Theorem 5.4. in \cite{paperd}. If $\epsilon=0$, the space satisfying the previous two conditions admit an isolated point, see Fig. \ref{figdiscde} left. 

\item Consider $(x,X),(y,Y)\in T \mathbb R^2$ with $\kappa=1$:

\begin{itemize}
\item  $x=(0,0)$; $X=e^{2\pi i}\in T_x\mathbb R^2$.
\item $y=(4-\epsilon,0)$; $0<\epsilon<4$; $Y=e^{2\pi i}\in T_y\mathbb R^2$. 
\end{itemize}

Intimately related to the previous observation is that the path $\gamma$ shown in Fig. \ref{motiv} (a non-embedded path) is not homotopic (while preserving the curvature bound throughout the deformation) to the line segment connecting $(x,X)$ to $(y,Y)$, see Corollary 7.13 in \cite{paperc}. In contrast, if $\epsilon\leq 0$, then these two paths are homotopic without violating the curvature bound \cite{paperd}. This fact is related with the existence of trapped regions \cite{paperc}. If a bound on curvature is not under consideration then, $\gamma$ and the line segment connecting $(x,X)$ to $(y,Y)$ are regular homotopic, see Fig. \ref{motiv}. 

%----------------------------------------------------------------------------------------
%	FIGURE 3
%----------------------------------------------------------------------------------------
\begin{figure}[h]
	\centering
	\includegraphics[width=.8\textwidth,angle=0]{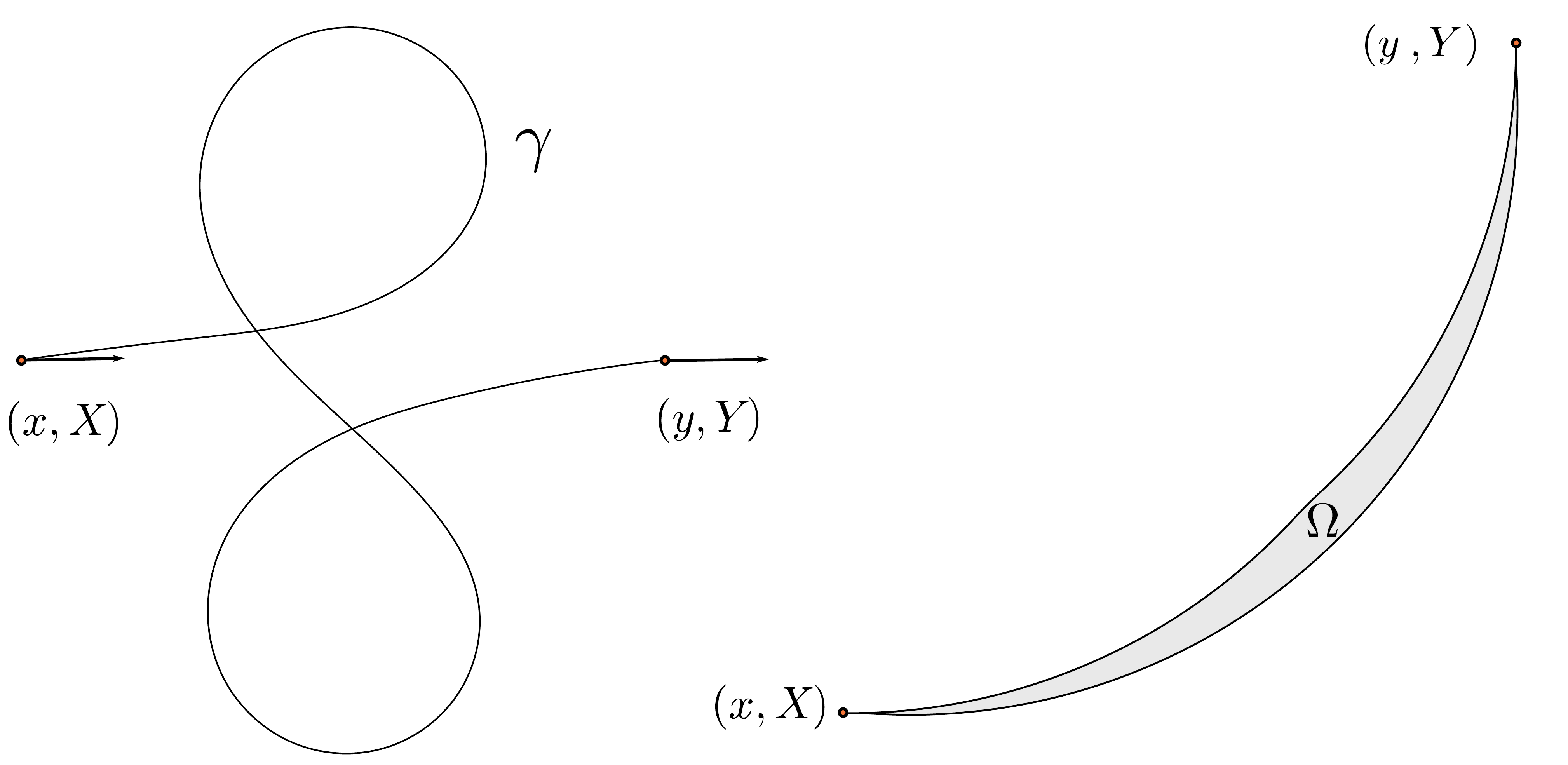}
	\caption{Right: A (zoomed out) trapped region $\Omega\subset \mathbb R^2$ obtained after perturbing the final position of a bounded curvature path being an isolated point, see Fig. \ref{figdiscde} left. Suddenly, the topology of the path space changes from a space of paths admitting an isolated point into a space of paths admitting a bounded isotopy class with non-empty interior. The elements of this isotopy class are only defined in $\Omega$. Left: An illustration of (5) in Example \ref{ex:1}. If $d(x,y)<4$, then $\gamma$ is not homotopic while preserving the curvature bound through the deformation to the line segment (the length minimizer) from $(x,X)$ to $(y,Y)$.}
	\label{motiv}
\end{figure}
\end{enumerate}
\end{example}

%----------------------------------------------------------------------------------------
%	SPACES OF BOUNDED CURVATURE PATHS
%----------------------------------------------------------------------------------------
\subsection{Spaces of bounded curvature paths}

\begin{definition} \label{admsp} Given $\mbox{\sc x,y}\in T{\mathbb R}^2$. The space of bounded curvature paths from {\sc x} to {\sc y} is denoted by $\Gamma(\mbox{\sc x,y})$.
\end{definition}

In this note we consider $\Gamma(\mbox{\sc x,y})$ with the topology induced by the $C^1$ metric. It is important to note that properties (among many others) such as types of connected components, or the number of local (global) minima in $\Gamma(\mbox{\sc x,y})$ depend on the endpoints in $T{\mathbb R}^2$ under consideration. 

%Next definition leads to a correspondence between an element in the 

Next, we make use of the fibre bundle structure of $T\mathbb R^2$ to describe families of spaces of bounded curvature paths.

\begin{definition}\label{famspa} Choose $\mbox{\sc x}\in T{\mathbb R}^2$ and $y \in \mathbb R^2$. Consider the family of pairs $\mbox{\sc x}, \mbox{\sc y}_\theta \in T\mathbb R^2$ with $\mbox{\sc y}_\theta=(y,Y_\theta)$; $ Y_\theta=e^{\theta i}\in T_y\mathbb R^2$, $\theta \in \mathbb R$. The one-parameter family of spaces of bounded curvature paths starting at $\mbox{\sc x}$ and finishing at $\mbox{\sc y}_\theta$ is called a {\it fiber} and is denoted by $\Gamma(\mbox{\sc x}, \mbox{\sc y}_\theta)$.
\end{definition}

Whenever we write: $\mbox{\sc x}, \mbox{\sc y}_\theta \in T\mathbb R^2$, $\theta \in \mathbb R$, we mean a family of pairs of endpoints so that: $\mbox{\sc x}\in T\mathbb R^2$ and $y\in \mathbb R^2$ are arbitrary but fixed while $\theta$ varies in the reals. Note that a space $\Gamma(\mbox{\sc x,y})$ is a representative of a family of spaces parametrized in the reals. 

\begin{definition}\label{gammaparameter}  Given $\mbox{\sc x}\in T{\mathbb R}^2$ we define:
$$\Gamma=\bigcup_{\substack{{y \in \mathbb R^2}\\ \theta \in \mathbb R}}\Gamma(\mbox{\sc x}, \mbox{\sc y}_\theta).$$
\end{definition}

In this note we develop a method for parametrizing the fibers in $\Gamma$ in terms of the types of connected components in $\Gamma(\mbox{\sc x}, \mbox{\sc y}_\theta)$, $\theta\in \mathbb R$. 

When a path is continuously deformed under parameter $p$ we reparametrize each of the deformed paths by its arc-length. In this fashion, $\gamma: [0,s_p]\rightarrow {\mathbb R}^2$ represents a deformed path at parameter $p$, with $s_p$ corresponding to its arc-length. 

\begin{definition}  \label{hom_adm} Given $\gamma,\eta \in \Gamma(\mbox{\sc x,y})$. A {\it bounded curvature homotopy}  between $\gamma: [0,s_0] \rightarrow  {\mathbb R^2}$ and $\eta: [0,s_1] \rightarrow  {\mathbb R^2}$ corresponds to a continuous one-parameter family of immersed paths $ {\mathcal H}_t: [0,1] \rightarrow \Gamma(\mbox{\sc x,y})$ such that:
\begin{itemize}
\item ${\mathcal H}_t(p): [0,s_p] \rightarrow  {\mathbb R}^2$ for $t\in [0,s_p]$ is an element of $\Gamma(\mbox{\sc x,y})$ for all $p\in [0,1]$.
\item $ {\mathcal H}_t(0)=\gamma(t)$ for $t\in [0,s_0]$ and ${\mathcal H}_t(1)=\eta(t)$ for $t\in [0,s_1]$.
\end{itemize}
\end{definition}

A bounded-curvature isotopy is a continuous one-parameter family of embedded bounded curvature paths. Two paths in $\Gamma(\mbox{\sc x,y})$ are {\it bounded-homotopic (bounded-isotopic)} if there exists a bounded curvature homotopy (isotopy) from one to the other. A {\it homotopy (isotopy) class} is a maximal path connected set in $\Gamma(\mbox{\sc x,y})$.

%\begin{definition} \label{admsp} Let $\Delta(\mbox{\sc x,y})$ be an isotopy class of paths in $\Gamma(\mbox{\sc x,y})$. And, $\mathcal B=\{\mbox{\sc x},\mbox{\sc y} \in T\mathbb R^2: \Delta(\mbox{\sc x,y})\neq \emptyset \}$.
%\end{definition}

\begin{definition} \label{admsp} Let $\Delta(\mbox{\sc x,y})$ be a non-empty bounded isotopy class of paths in $\Gamma(\mbox{\sc x,y})$. Let $\mathcal B$ denote the set of pairs $\mbox{\sc x,y} \in T\mathbb R^2$ for which $\Gamma(\mbox{\sc x,y})$ possesses such a bounded isotopy class. 
\end{definition}

%THIS MAY NOT BE NECESSARY There exists $\mbox{\sc x,y} \in T\mathbb R^2$ so that such bounded isotopy classes $\Delta(\mbox{\sc x,y})$ exist, hence $\mathcal B$ is non-empty.

In \cite{paperc} we proved that $\mathcal B\neq \emptyset$ by establishing the existence of non-empty bounded isotopy classes. Whenever we refer to $\Delta(\mbox{\sc x,y})$ we imply that $\Delta(\mbox{\sc x,y})$ is non-empty. In this note, we give necessary and sufficient conditions so that $\mbox{\sc x},\mbox{\sc y} \in T\mathbb R^2$ is an element in $\mathcal B$. As a consequence, we answer a question raised by Dubins in pp. 480 in \cite{dubins 2}. We establish that $\mathcal B$ is a bounded neither open nor closed subset in $T\mathbb R^2$, see Theorem \ref{noopnoclofib} and Corollary \ref{noopnoclo}.

%THIS MAY NOT BE NECESSARY A core result in \cite{paperc}, called the $S$-Lemma, relates the bound on curvature with the turning map and its extremals, to ultimately conclude that there exist $\mbox{\sc x,y}\in T\mathbb R^2$, so that $\Delta(\mbox{\sc x,y})\neq \emptyset$.

%----------------------------------------------------------------------------------------
%	PROXIMITY OF ENDPOINTS
%----------------------------------------------------------------------------------------
\subsection{Proximity of endpoints}\label{proxcon} 

Here we analyze the configurations of distinguished pairs of circles in $\mathbb R^2$. This approach permits us to reduce the configurations of endpoints in $T\mathbb R^2$ into a finite number of cases up to isometries. 

Consider $\mbox{\sc x}\in T\mathbb R^2$. Let $\mbox{\sc C}_ l(\mbox{\sc x})$ be the unit radius circle tangent to $x$ and to the left of $X$. The meaning of $\mbox{\sc C}_ r(\mbox{\sc x})$, $\mbox{\sc C}_ l(\mbox{\sc y})$ and $\mbox{\sc C}_ r(\mbox{\sc y})$ should be obvious. These circles are called {\it adjacent circles}. Denote the centers of the adjacent circles with lowercase letters. So, the center of $\mbox{\sc C}_ l(\mbox{\sc x})$ is $c_l(\mbox{\sc x})$, see Fig. \ref{fig:Cr} right. The other cases are analogous. 

We concentrate on the following configurations for the adjacent circles.

\begin{equation} d(c_l(\mbox{\sc x}),c_l(\mbox{\sc y}))\geq 4 \quad \mbox{and}\quad d(c_r(\mbox{\sc x}),c_r(\mbox{\sc y}))\geq4 \label{con_a}\tag{i}\end{equation}
\vspace{-1.5em}
 \begin{equation} d(c_l(\mbox{\sc x}),c_l(\mbox{\sc y}))< 4 \quad \mbox{and}\quad d(c_r(\mbox{\sc x}),c_r(\mbox{\sc y}))\geq 4 \label{con_b}\tag{ii} \end{equation}
 \vspace{-1.5em}
  \begin{equation} d(c_l(\mbox{\sc x}),c_l(\mbox{\sc y}))\geq4 \quad \mbox{and}\quad d(c_r(\mbox{\sc x}),c_r(\mbox{\sc y}))< 4  \label{con_b'}\tag{iii} \end{equation}
  \vspace{-1.5em}
   \begin{equation} d(c_l(\mbox{\sc x}),c_l(\mbox{\sc y}))< 4 \quad \mbox{and}\quad d(c_r(\mbox{\sc x}),c_r(\mbox{\sc y}))< 4 \label{con_c}\tag{iv} 
         \end{equation}
         
   The conditions (i)-(iv) have being used in different contexts through \cite{paperb,papera, paperc, paperd}. They give information about the topology and geometry of $\Gamma(\mbox{\sc x,y})$. Note that as planar configurations, (ii) and (iii) are equivalent up to isometries. 

%----------------------------------------------------------------------------------------
%	TRAPPED REGIONS AND ISOTOPY CLASSES
%----------------------------------------------------------------------------------------
\subsection{Trapped regions and bounded isotopy classes}\label{construct} 

In Theorem 5.4 in \cite{paperd} we proved that for certain $\mbox{\sc x,y}\in T{\mathbb R}^2$, the associated space $\Gamma(\mbox{\sc x,y})$ admits a bounded isotopy class $\Delta(\mbox{\sc x,y})$. It turns out that paths in $\Delta(\mbox{\sc x,y})$ are defined exclusively in a bounded region $\Omega\subset \mathbb R^2$. The shape of $\Omega$ depends on the initial and final positions and directions in $T{\mathbb R}^2$, see Fig. \ref{regparam}. 

For a precise explanation on how these regions $\Omega\subset \mathbb R^2$ are constructed, we strongly suggest the reader refer to Section 4 in \cite{paperc}. We call these regions {\bf trapped regions}. 

It is important to note that:
\begin{itemize}

\item Embedded paths in $\Omega$ cannot be deformed without violating the curvature bound to a path with a self-intersection, see Corollary 7.13 in \cite{paperc}.

\item Embedded paths in $\Omega$ are not bounded-homotopic to paths having a point not in $\Omega$, see Theorem 8.1 in \cite{paperc}. 

\item The proof of the existence of isolated points in spaces of bounded curvature paths was given in Theorem 3.9 in \cite{paperc}. These correspond to arcs of a unit circle of length less to $\pi$, called {\bf {\sc c} isolated points}. Similarly, a concatenation of two arcs of unit circle, each of length less to $\pi$, are called {\bf {\sc cc} isolated points}. Isolated points in $\Gamma(\mbox{\sc x,y})$ are bounded isotopy classes with empty interior, see Fig. \ref{figdiscde} left, and Fig. \ref{figgenpos} left. In addition, bounded curvature paths of length zero are also isolated points. This observation becomes interesting after recalling the concept of simple connectedness. Closed bounded curvature paths are not bounded-homotopic to a single point. 
\end{itemize}

\begin{remark}\label{rem:empty} Suppose that for $\mbox{\sc x,y}\in T{\mathbb R}^2$ we have that $\Gamma(\mbox{\sc x,y})$ does not admit a bounded isotopy class $\Delta(\mbox{\sc x,y})$. Then, embedded trapped paths cannot exist. We adopt the notation $\Delta(\mbox{\sc x,y})$, rather than $\Delta(\Omega)$ as we did in \cite{paperc, paperd}, since our emphasis now is on the endpoints rather than the regions $\Omega\subset \mathbb R^2$. We prefer to write $\Omega$ instead of $\Omega(\mbox{\sc x,y})$.
\end{remark}

The classification theorem for the homotopy classes in \cite{paperd} required the proximity conditions A, B, C, and D, see \cite{paperc,paperd}. Next, we redefine conditions C and D in terms of the existence of bounded isotopy classes, see Fig. \ref{figproxcondabcd}.

%After the classification theorem for the homotopy classes in \cite{paperd} we can redefine proximity conditions A, B, C, and D in \cite{paperc,paperd} in terms of the existence of bounded isotopy classes, see Fig. \ref{figproxcondabcd}.

%----------------------------------------------------------------------------------------
%	FIGURE 4
%----------------------------------------------------------------------------------------
\begin{figure}[h]
	\centering
	\includegraphics[width=.8\textwidth,angle=0]{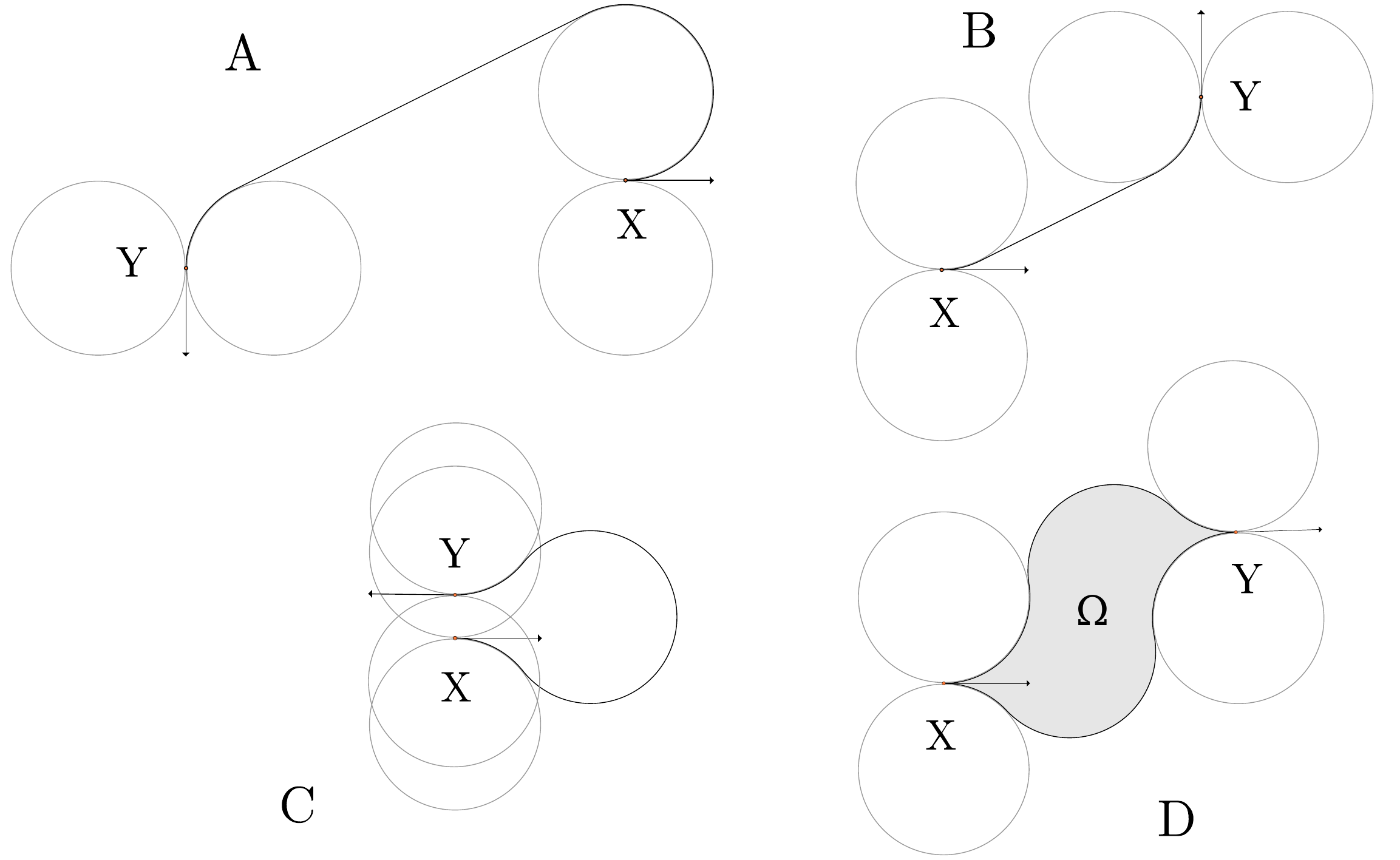}
	\caption{Examples of bounded curvature paths in spaces satisfying conditions A, B, C and D.}
	\label{figproxcondabcd}
\end{figure}

\begin{definition}\label{procon}
If $\mbox{\sc x,y}\in T{\mathbb R}^2$ satisfies:
\begin{itemize}
\item (i) then $\Gamma({\mbox{\sc x,y}})$  is said to satisfy proximity condition {\sc A}.
\item (ii) or (iii) then $\Gamma({\mbox{\sc x,y}})$ is said to satisfy proximity condition {\sc B}.
\item (iv) and there is no bounded isotopy class $\Delta({\mbox{\sc x,y}})$ then $\Gamma({\mbox{\sc x,y}})$ is said to satisfy proximity condition {\sc C}.
\item (iv) and there exists a bounded isotopy class $\Delta({\mbox{\sc x,y}})$ then $\Gamma({\mbox{\sc x,y}})$ is said to satisfy proximity condition {\sc D}.  
\end{itemize}
\end{definition}

\noindent In Theorem \ref{maincensus1} we clarify for what $\mbox{\sc x,y}\in T\mathbb R^2$ we have that $\Delta(\mbox{\sc x,y})\subset \Gamma(\mbox{\sc x,y})$. To this end, we group spaces of bounded curvature paths in terms of the type of connected components that they have.

%----------------------------------------------------------------------------------------
%	AN UNDERLYING DISCRETE STRUCTURE
%----------------------------------------------------------------------------------------
\section{An underlying discrete structure}\label{underlying} \label{Crl} 

Next, we describe the coordinates of distinguished points in $\mathbb R^2$. The configurations of these points reveal interesting features of $\Gamma(\mbox{\sc x,y})$, $\mbox{\sc x,y}\in T\mathbb R^2$. In particular, these points completely characterize the regions $\Omega\subset \mathbb R^2$ whenever they exist. 

%----------------------------------------------------------------------------------------
%	FIGURE 5
%----------------------------------------------------------------------------------------
\begin{figure}[h]
	\centering
	\includegraphics[width=1\textwidth,angle=0]{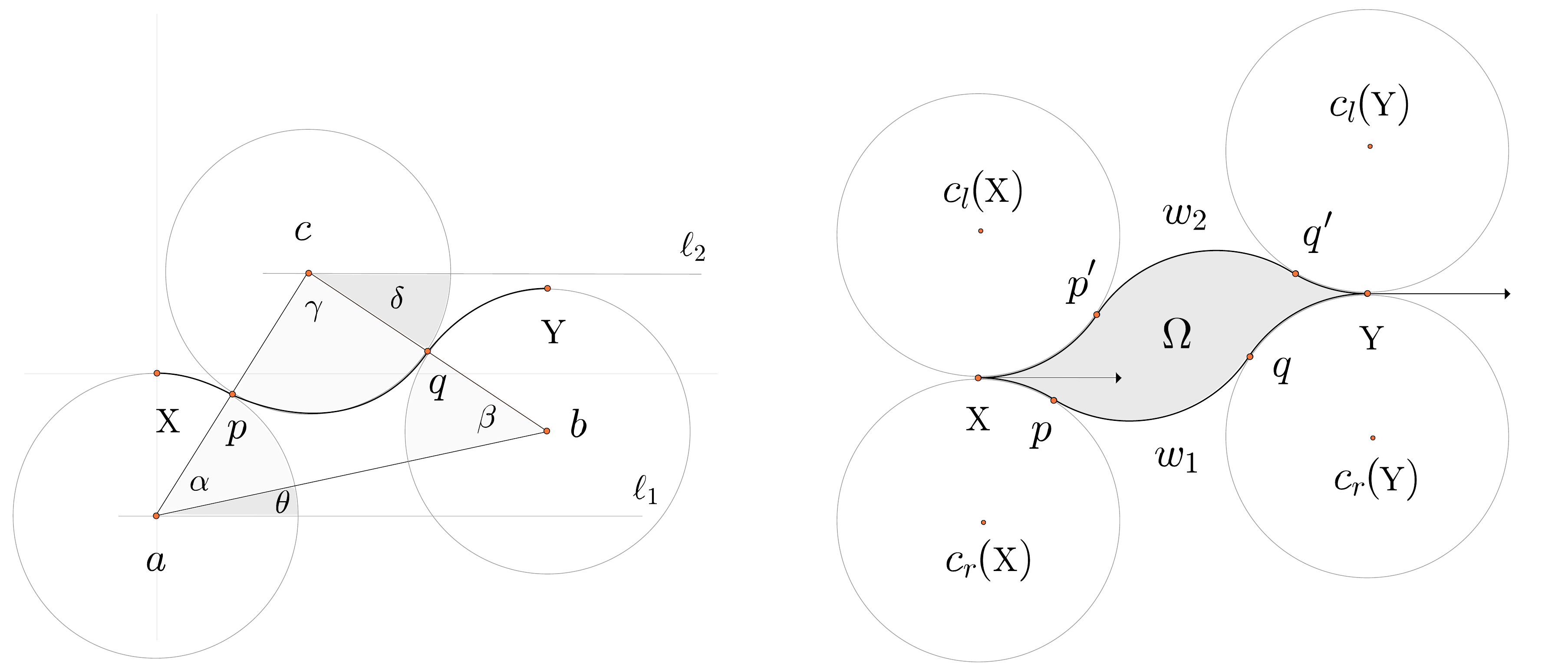}
	\caption{Right: Notation associated with $\Omega\subset \mathbb R^2$. Left: The angles involved when computing the coordinates of $p,q\in \mathbb R^2$.}
	\label{fig:Cr}
\end{figure}

Consider unit radius circles $A$ and $B$ with centers $a=(a_1,a_2)$ and $b=(b_1,b_2)$ respectively. Consider a unit radius circle $C$ with center $c$ tangent to $A$ and $B$ at $p$ and $q$ respectively, see Figure \ref{fig:Cr} left. Also, set $\mbox{\sc x,y}\in T\mathbb R^2$ so that a {\sc ccc} path is obtained. Suppose the coordinates of $a$ and $b$ are known. Next we determine the coordinates of the points $p$ and $q$. 

Consider the triangle whose vertices are $a$, $b$, and $c$. Denote by $\theta$ the smallest angle made by the line passing through $a$ and $b$ and the horizontal axis according to Remark \ref{coo}, see Figure \ref{fig:Cr} left. Here $\ell_1$ and $\ell_2$ are parallel to the horizontal axis. Denote by $\delta$ the smallest angle made by the line passing through $b$ and $c$ and the horizontal axis. It is easy to see that $d(a,c)=d(c,b)=2$. 

After applying the law of cosines we immediately obtain that:

\begin{equation*} 
\alpha =\arccos\bigg({\frac{\sqrt{(b_1-a_1)^2+(b_2-a_2)^2}}{4}}\bigg)
\end{equation*}
\begin{equation*} 
\theta = \arctan\bigg({\frac{b_2-a_2}{b_1-a_1}}\bigg)
\end{equation*}
\begin{equation*} 
\delta = \arctan\bigg({\frac{b_2-c_2}{b_1-c_1}}\bigg)
\end{equation*}
\begin{equation*} 
	\label{eq:cruv}	
	c=(a_1 + 2\cos(\alpha+\theta), a_2 + 2\sin{(\alpha+\theta}))
\end{equation*}
\vspace{0em}
\begin{equation} 
	\label{eq:i1}	
	p = (a_1 + \cos({\alpha+\theta}), a_2 + \sin({\alpha+\theta}))
\end{equation}
\vspace{0em}
\begin{equation} 
	\label{eq:i3}	
	q= (b_1 + \cos{\delta},b_2 + \sin{\delta})
\end{equation}
\vspace{0.1em}

By letting $A=\mbox{\sc C}_ r(\mbox{\sc x})$ and  $B=\mbox{\sc C}_ r(\mbox{\sc y})$ we find explicit formulas for the point $p$ between $A$ and $C$ and $q$ between $C$ and $B$, see Fig. \ref{fig:Cr} left. Observe that the coordinates of $c_r(\mbox{\sc x})$ and $c_ r(\mbox{\sc y})$ are easily obtained since $\mbox{\sc x,y}\in T\mathbb R^2$ are given. 

Analogously, by letting $A=\mbox{\sc C}_ l(\mbox{\sc x})$ and  $B=\mbox{\sc C}_ l(\mbox{\sc y})$, and by applying the same reasoning as before, we find formulas for the points $p'$ between $A$ and $C'$ and $q'$ between $C'$ and $B$ (see Fig. \ref{fig:Cr} right):

\begin{equation*} 
\alpha' =\arccos\bigg({\frac{\sqrt{(b_1-a_1)^2+(b_2-a_2)^2}}{4}}\bigg)
\end{equation*}
\begin{equation*} 
\theta' = \arctan\bigg({\frac{b_2-a_2}{b_1-a_1}}\bigg)
\end{equation*}
\begin{equation*} 
\delta' = \arctan\bigg({\frac{b_2-c_2}{b_1-c_1}}\bigg)
\end{equation*}
\begin{equation*} 
	\label{eq:cluv}	
	c'=(a_1 + 2\cos(\alpha'+\theta'), a_2 + 2\sin{(\alpha'+\theta'}))
\end{equation*}
\vspace{0em}
\begin{equation} 
	\label{eq:i2}	
	p' = (a_1 + \cos({\alpha'+\theta'}), a_2 + \sin({\alpha'+\theta'}))
\end{equation}
\vspace{0em}
\begin{equation} 
	\label{eq:i4}	
	q' = (b_1 + \cos{\delta'},b_2 + \sin{\delta'})
\end{equation}

\begin{definition}\label{boundomega} Let: 
\begin{itemize}
\item $w_1$ be the {\sc rlr} path consisting of an arc from $x$ to $p$ in $\mbox{\sc C}_ r(\mbox{\sc x})$; an arc from $p$ to $q$ in $C$; and an arc from $q$ to $y$ in $\mbox{\sc C}_ r(\mbox{\sc y})$, see equations (\ref{eq:i1}) and (\ref{eq:i3}).
\item $w_2$ be the {\sc lrl} path consisting of an arc from $x$ to $p'$ in $\mbox{\sc C}_ l(\mbox{\sc x})$; arc from $p'$ to $q'$ in $C'$; and an arc from $p'$ to $y$ in $\mbox{\sc C}_ l(\mbox{\sc y})$, see equations (\ref{eq:i2}) and (\ref{eq:i4}).
\end{itemize}
\end{definition}

 Next we make use of the formaulae (\ref{eq:i1})-(\ref{eq:i4}) to characterize $\Omega\subset \mathbb R^2$ in terms of the coordinates of distinguished points.

\begin{definition}\label{omegadfn} Assume $\Gamma(\mbox{\sc x}, \mbox{\sc y})$ satisfies condition {\sc D}. Let $\Omega\subset \mathbb R^2$ be the bounded region whose boundary is given by the union of $w_1$ and $w_2$ in Definition \ref{boundomega}, see Fig. \ref{fig:Cr} right. In this case we say that $\mbox{\sc x,y}\in T\mathbb R^2$ {\it carries a region}.
\end{definition}

%----------------------------------------------------------------------------------------
%	MOTIVATION THROUGH EXAMPLES
%----------------------------------------------------------------------------------------
\section{Motivation through examples }\label{moduli}

In narrative terms, here we present facts in reverse-chronology. By considering this strategy, we are telling the reader ``the end of the story'' through various examples, with the intention to motivate the more technical steps and proofs. 

We study the fibers in $\Gamma$ by fixing $\mbox{\sc x}=(x,X)\in T\mathbb R^2$ and a final position $y\in \mathbb R^2$ while varying the final direction $Y_{\theta}\in T_y\mathbb R^2$, $\theta \in \mathbb R$. In Section \ref{classrangesec} we construct a function, called the {\bf class range}, that assigns to each final position $y\in \mathbb R^2$ a non-negative real number. This number is called the class value and gives the range $\theta$ can vary so that the spaces $\Gamma(\mbox{\sc x}, \mbox{\sc y}_{\theta})$ have the same types of connected components. 

Firstly, we would like to point out that for a fixed $\theta\in \mathbb R$, the spaces $\Gamma(\mbox{\sc x}, \mbox{\sc y}_{k\theta})$ and $\Gamma(\mbox{\sc x}, \mbox{\sc y}_{j\theta})$ may eventually be different for $j\neq k\in \mathbb Z$. Secondly, consider $\mbox{\sc x}, \mbox{\sc y}_{\theta}\in T\mathbb R^2$, so that $\theta=\pm \pi$. Since the initial and final tangent vectors are parallel having opposite sense, the pairs $\mbox{\sc x}, \mbox{\sc y}_{\pm \pi}\in T\mathbb R^2$ do not carry a region $\Omega\subset \mathbb R^2$. This is due to the existence of parallel tangents, see \cite{papere}. 

\begin{definition}\label{defw}Consider $\mbox{\sc x}, \mbox{\sc y}_\theta \in T\mathbb R^2$, $\theta\in (-\pi,\pi)$, so that $\Gamma(\mbox{\sc x}, \mbox{\sc y}_{\theta})$ satisfy proximity condition {\sc D}. Let
\begin{itemize}
 \item $\omega_- $ be the smallest value in $(-\pi,\pi)$ so that there exits a bounded isotopy class $\Delta(\mbox{\sc x}, \mbox{\sc y}_\theta)$;
 
 \item $\omega_+ $ be the greatest value in $(-\pi,\pi)$ so that there exits a bounded isotopy class $\Delta(\mbox{\sc x}, \mbox{\sc y}_\theta)$.
 
   \end{itemize}
   An interval whose endpoints are $\omega_-$ and $\omega_+$ is denoted by $I(y)$. We refer to $\omega_- $ and $\omega_+$ as critical angles.
\end{definition}

\begin{remark}\hfill 
\begin{itemize} 
\item The critical angles $\omega_-$ and $\omega_+ $ depend on $y\in \mathbb R^2$ since in Remark \ref{coo} we established that $(x,X)\in T\mathbb R^2$ is fixed. Sometimes we write $\omega_-=\omega_-(y)$ and $\omega_-=\omega_-(y)$. 
\item From the way we construct the class range function (see Definition \ref{classrange+}) the existence of $\omega_-$ and $\omega_+$ is guaranteed.
\end{itemize}
\end{remark}

The examples in \ref{paramex} and the illustrations in Fig. \ref{regparam} have been obtained computationally \cite{dubinsexplorer} by evaluating the class range function in Definition \ref{classrange+} via equations (\ref{eq:Wmin}) and (\ref{eq:Wmax}). Throughout this note, whenever we consider examples obtained computationally, angles will be measured in degrees. The following ideas will be formalized in Sections \ref{classrangesec} and \ref{classdomain}, compare Definition \ref{def:spaces}.

%----------------------------------------------------------------------------------------
%	                         PARAMETRIZING FIBERS
%----------------------------------------------------------------------------------------
\begin{example}[\bf parametrizing fibers]\label{paramex}
In Fig. \ref{regparam} top, we show a sequence for the variation of $\mbox{\sc x}, \mbox{\sc y}_\theta\in T\mathbb R^2$ while $\theta$ ranges over $I(y) \subsetneq (-180^{\circ},180^{\circ})$ illustrating the following example.

Consider $x=(0,0)\in \mathbb R^2$, $X=(1,0)\in T_x\mathbb R^2$, $y=(2.82,0)\in\mathbb R^2$,$Y_{\theta}=e^{i\theta}\in T_y\mathbb R^2$. We determine that $I(y)=[-109.47^{\circ},109.47^{\circ}]$ and conclude that $\Gamma(\mbox{\sc x}, \mbox{\sc y}_\theta)$ admits (from left to right and counterclockwise) spaces of bounded curvature paths being:

\begin{itemize} 
\item An isolated point for $\theta=-109.47^{\circ}$, see Theorem 3.9 in \cite{paperc}. 
\item There exists a bounded isotopy class $\Delta(\mbox{\sc x}, \mbox{\sc y}_\theta)$ for $\theta \in (-109.47^{\circ},109.47^{\circ})$ i.e., a one-parameter family of bounded isotopy classes, see Theorem \ref{existvect}. Also see Theorem 8.1 in \cite{paperc}.
\item An isolated point for $\theta=109.47^{\circ}$. 
\item If $\theta \notin [-109.47^{\circ},109.47^{\circ}]$ then there is no bounded $\Delta(\mbox{\sc x}, \mbox{\sc y}_\theta)$.
\end{itemize}

In this case we say that $\Gamma(\mbox{\sc x}, \mbox{\sc y}_\theta)$ is a {\bf fiber of type I}.

%----------------------------------------------------------------------------------------
%	FIGURE 6
%----------------------------------------------------------------------------------------
\begin{figure}[h]
	\centering
	\includegraphics[width=1\textwidth,angle=0]{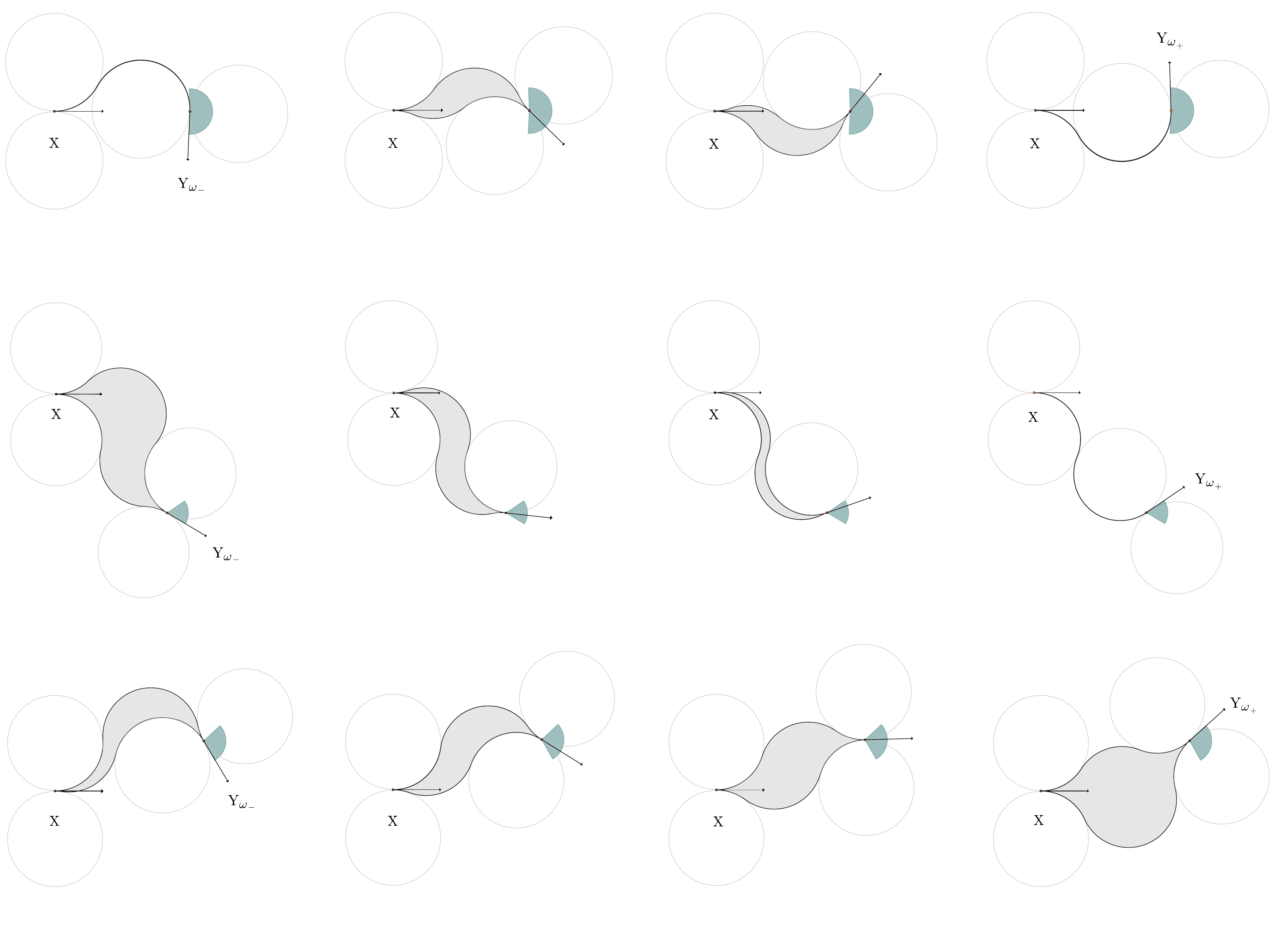}
	\caption{The grey regions are examples of $\Omega\subset \mathbb R^2$. We illustrate the examples in Remark \ref{paramex} computed and plotted according to Definition \ref{classrange+} via equations (\ref{eq:Wmin}) and (\ref{eq:Wmax}). The range where $\theta$ can vary is depicted in dark green.}
	\label{regparam}
\end{figure}

In Fig. \ref{regparam} middle we show a sequence for the variation of $\mbox{\sc x}, \mbox{\sc y}_\theta\in T\mathbb R^2$ while $\theta$ ranges over $I(y) \subsetneq (-180^{\circ},180^{\circ})$ illustrating the following example.

Consider $x=(0,0)\in \mathbb R^2$, $X=(1,0)\in T_x\mathbb R^2$, $y=(2.5,-2)\in \mathbb R^2$, $Y_{\theta}=e^{i\theta}\in T_y\mathbb R^2$. We determine that $I(y)=[-48.36^{\circ},30.30^{\circ})$ and conclude that $\Gamma(\mbox{\sc x}, \mbox{\sc y}_\theta)$ contains (from left to right and counterclockwise) spaces of bounded curvature paths such that:

\begin{itemize}
\item There exists a bounded isotopy class $\Delta(\mbox{\sc x}, \mbox{\sc y}_\theta)$ for $\theta \in [-48.36^{\circ},30.30^{\circ})$ i.e., a one-parameter family of bounded isotopy classes.
\item An isolated point for $\theta =30.30^{\circ}$. 
\item If $\theta \notin  [-48.36^{\circ},30.30^{\circ})$ then there is no bounded $\Delta(\mbox{\sc x}, \mbox{\sc y}_\theta)$.
\end{itemize}

In this case we say that $\Gamma(\mbox{\sc x}, \mbox{\sc y}_\theta)$ is a  {\bf fiber of type II}.

In Fig. \ref{regparam} bottom we show a sequence for the variation of $\mbox{\sc x}, \mbox{\sc y}_\theta\in T\mathbb R^2$ while $\theta$ ranges over $I(y) \subsetneq (-180^{\circ},180^{\circ})$ illustrating the following example.

 Consider $x=(0,0)\in \mathbb R^2$, $X=(1,0)\in T_x\mathbb R^2$, $y=(3,0.5)\in \mathbb R^2$ and $Y_{\theta}=e^{i\theta}\in T_y\mathbb R^2$ we determine that $I(y)=[-80.42^{\circ},60.55^{\circ}]$ and conclude that $\Gamma(\mbox{\sc x}, \mbox{\sc y}_\theta)$ contains (from left to right and counterclockwise) spaces of bounded curvature paths being:

\begin{itemize}
\item There exists a bounded isotopy class $\Delta(\mbox{\sc x}, \mbox{\sc y}_\theta)$ for $\theta \in  [-80.42^{\circ},60.55^{\circ}]$  i.e., a one-parameter family of bounded isotopy classes.
\item If  $\theta \notin [-80.42^{\circ},60.55^{\circ}]$ then there is no bounded $\Delta(\mbox{\sc x}, \mbox{\sc y}_\theta)$.
\end{itemize}

In this case we say that $\Gamma(\mbox{\sc x}, \mbox{\sc y}_\theta)$ is a  {\bf fiber of type III}.
\end{example}

There are two more types of fibers, these will be discussed in Definition \ref{def:spaces}.

 In Section \ref{classdomain}  we characterize a region $B\subset \mathbb R^2$ so that the class range is well defined. This plane region corresponds exactly to the location for the final positions $y\in \mathbb R^2$ so that $\Gamma(\mbox{\sc x}, \mbox{\sc y}_\theta)$ admits a family of bounded isotopy classes of bounded curvature paths.
 
Next we explain two types of configurations that will be of relevance when determining the extreme values of the class range function. 
%----------------------------------------------------------------------------------------
%	TOPOLOGICAL TRANSITIONS
%----------------------------------------------------------------------------------------
\section{Topological transitions}\label{crit}

 Consider a {\sc cc} isolated point as shown in Fig. \ref{figgenpos} left. After a small clockwise continuous perturbation on the final direction (while fixing the final position) the resultant endpoints define a space that does not admit an isolated point, see Fig. \ref{figgenpos} middle and right. This is true since the paths at middle and right (the length minimisers in their respective space) are parallel homotopic to paths of arbitrary length due to the existence of parallel tangents, see Corollary 3.4 and Proposition 3.8 in \cite{papere}. By Corollary 7.13 in \cite{paperc} these paths are not elements in bounded isotopy classes. 

It is fairly easy to see that a small counterclockwise perturbation on the final direction of the {\sc cc} isolated point in Fig. \ref{figgenpos} leads to spaces admitting a bounded isotopy class.

\begin{itemize}
\item[(1)] For certain fibers, the {\sc cc} isolated points are transitions between spaces with different types of connected components.
 \end{itemize}
 
 The previous observations say implicitly that for certain fibers the critical values $\omega_-$ and $\omega_+$ are achieved at spaces admitting {\sc cc} isolated points.

%----------------------------------------------------------------------------------------
%	FIGURE 7
%----------------------------------------------------------------------------------------
\begin{figure}[h]
	\centering
	\includegraphics[width=.7\textwidth,angle=0]{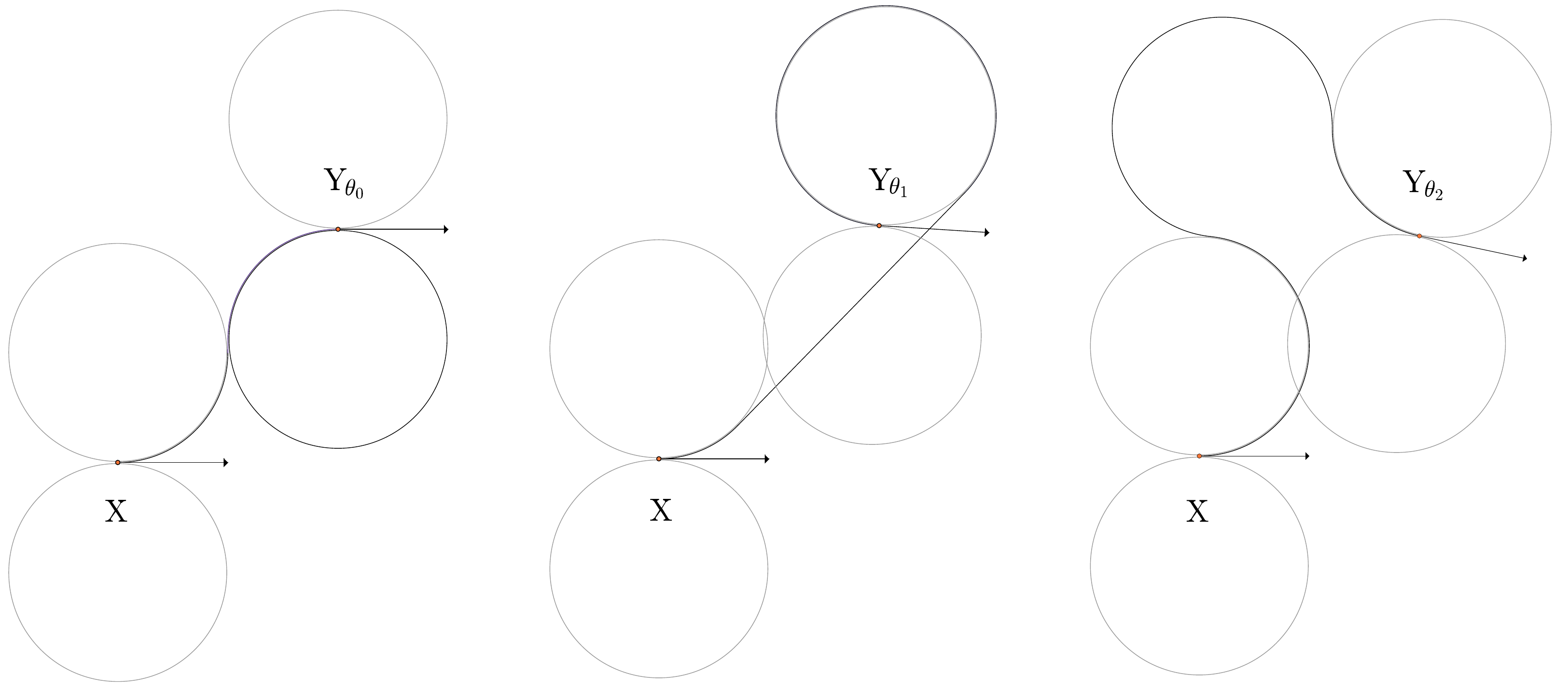}
	\caption{Two types of discontinuities. When varying the final vector of an isolated point (at the left) we obtain a length discontinuity. The third path shows that length discontinuities not only happen when perturbing directions of isolated points. Here $x=(0,0)\in \mathbb R^2$, $X=(1,0)\in T_x\mathbb R^2$, $y=(2,2)\in \mathbb R^2$, and $e^{i \theta }=Y_\theta \in T_y\mathbb R^2$ with $\theta_0=0^{\circ}$,  $\theta_1=-6^{\circ}$,  $\theta_2=-12^{\circ}$.}
	\label{figgenpos}
\end{figure}

Recall that a necessary condition for the existence of a bounded isotopy class $\Delta(\mbox{\sc x}, \mbox{\sc y})$ is that $\mbox{\sc x,y}\in T\mathbb R^2$ satisfy: 
$$d(c_l(\mbox{\sc x}),c_l(\mbox{\sc y}))< 4 \quad \mbox{and}\quad d(c_r(\mbox{\sc x}),c_r(\mbox{\sc y}))< 4.$$

This is easy to see since: if $d(c_l(\mbox{\sc x}),c_l(\mbox{\sc y}))\geq 4$, then a unit disk can be placed in the line joining $c_l(\mbox{\sc x})$ to $c_l(\mbox{\sc y})$ without overlapping with $C_l(\mbox{\sc x})$ or $C_l(\mbox{\sc y})$. This implies that bounded curvature paths may escape $\Omega\subset \mathbb R^2$ (after applying an operation of type {\sc II} in \cite{paperd} to the length minimiser in $\Omega$) contradicting Theorem 8.1 in \cite{paperc}. For details we recommend the reader refer to Section 4 in \cite{paperc}. The same applies for $d(c_r(\mbox{\sc x}),c_r(\mbox{\sc y}))\geq 4$.

\begin{itemize}
\item[(2)]  For certain fibers, the condition
\begin{equation} d(c_l(\mbox{\sc x}),c_l(\mbox{\sc y}))=4 \quad \mbox{and}\quad d(c_r(\mbox{\sc x}),c_r(\mbox{\sc y}))=4 \label{con_t}\end{equation}
 \end{itemize}
is considered as a transition between spaces with different types of connected components. We proved in Theorem 5.3 in \cite{paperd} that spaces $\Gamma(\mbox{\sc x},\mbox{\sc y})$ satisfying proximity condition A, that is:
\begin{equation*} d(c_l(\mbox{\sc x}),c_l(\mbox{\sc y}))\geq4 \quad \mbox{or}\quad d(c_r(\mbox{\sc x}),c_r(\mbox{\sc y}))\geq 4 \label{con_g}\end{equation*}

do not admit isotopy classes. Therefore, $\Omega=\emptyset$.

%----------------------------------------------------------------------------------------
%	ANGULAR FORMULAE
%----------------------------------------------------------------------------------------
\section{Angular formulae}\label{trans}
We describe two types auxiliary triangles that allow us to obtain (via continuous variations of their angles) the values of the class range function.  These triangles are constructed out of information obtained from the given endpoints in $T\mathbb R^2$. We establish a correlation between the angle variation in these auxiliary triangles and the types of connected components in $\Gamma(\mbox{\sc x}, \mbox {\sc y}_\theta)$, $\mbox{\sc x},\mbox{\sc y}_\theta\in T\mathbb R^2$, for each $\theta \in (-\pi,\pi)$.  

Next we consider fibers whose critical values $\omega_-$ and $\omega_+$ are achieved at spaces admitting {\sc cc} isolated points as disscused in (1) in Section \ref{crit}.

%----------------------------------------------------------------------------------------
%	SHORT TRIANGLES
%----------------------------------------------------------------------------------------
\subsection{Short triangles}\label{short}  Suppose $\mbox{\sc x}, \mbox{\sc y}_\theta \in T\mathbb R^2$, $\theta \in (-\pi,\pi)$ is such that for some $\theta\in (-\pi,\pi)$ the adjacent circles $\mbox{\sc C}_l(\mbox{\sc x})$ and $\mbox{\sc C}_r(\mbox{\sc y}_{\theta})$ intersect at a single point. We have that $\theta=\omega_-$ or $\theta=\omega_+$, see Fig. \ref{shortriang1B}. For $\theta=\omega_-$, construct a triangle whose vertices are $c_l({\mbox{\sc x}}), c_r(\mbox{\sc y}_{\omega_-})$ and $y$. It is immediate that $d(c_l({\mbox{\sc x}}), c_r(\mbox{\sc y}_{\omega_-}))=2$ and that $d(c_r(\mbox{\sc y}_{\omega_-}),y)=1$. For $\theta=\omega_+$, construct the triangle whose vertices are $c_r({\mbox{\sc x}}), c_l(\mbox{\sc y}_{\omega_+})$ and $y$, see Fig \ref{shortriang2}. It is immediate that $d(c_r({\mbox{\sc x}}), c_l(\mbox{\sc y}_{\omega_+}))=2$ and that $d(c_l(\mbox{\sc y}_{\omega_+}),y)=1$. 

The obvious observation: A triangle with sides of length $1$ and $2$ cannot have a third side of length greater to $3$ leads us to analyze the transitions in (2) in Section \ref{crit}.

%----------------------------------------------------------------------------------------
%	LONG TRIANGLES
%----------------------------------------------------------------------------------------
\subsection{Long triangles}\label{long} Consider $\mbox{\sc x}, \mbox{\sc y}_\theta \in T\mathbb R^2$, $\theta \in (-\pi,\pi)$ so that the adjacent circles $\mbox{\sc C}_l(\mbox{\sc x})$ and $\mbox{\sc C}_r(\mbox{\sc y}_{{\theta}})$ do not intersect. We construct the triangle whose vertices are $c_l({\mbox{\sc x}}), c_l(\mbox{\sc y}_\theta)$ and $y$, see Fig. \ref{fig:Triangle3}. Note that $d(c_l(\mbox{\sc y}_\theta),y)=1$, and $d(c_l({\mbox{\sc x}}), y)>3$, see Fig. \ref{fig:Triangle3}. 

In case the adjacent circles $\mbox{\sc C}_r(\mbox{\sc x})$ and $\mbox{\sc C}_l(\mbox{\sc y}_\theta)$ do not intersect, we construct the triangle whose vertices are $c_r({\mbox{\sc x}}), c_r(\mbox{\sc y}_\theta)$ and $y$, see Fig. \ref{fig:Triangle4}. Note that $d(c_r(\mbox{\sc y}_\theta),y)=1$, and $d(c_r({\mbox{\sc x}}), y)>3$. 

Next, we look closely at short and long triangles. Their sides are denoted in capital letters while the length of their sides are denoted in lowercase i.e., the side $S_i$ has length $s_i$. 
 
We keep a certain degree of detail for {short triangles} and subsequently reduce the details in the discussions for {long triangles} assuming the analogy in ideas and notation with short triangles.

%----------------------------------------------------------------------------------------
%	 ANGULAR FORMULA FOR SHORT TRIANGLES
%----------------------------------------------------------------------------------------
\subsection{Angular formulae for short triangles} \label{angformshort}
To avoid confusion, the abscissa and ordinate in the coordinate system in Remark \ref{coo} are denoted by $u$-axis and $v$-axis respectively. 

Using the notaion from Fig. \ref{shortriang1B} left, consider the triangle whose vertices are $c_l({\mbox{\sc x}}), c_r(\mbox{\sc y})$ and $y$. It is easy to see that this triangle has sides of length $a_1=2$, $b_1=1$, and $c_1=d(c_l(\mbox{\sc x}),y)$. In addition, since the endpoints are given, we can easily obtain the coordinates of the vertices of the triangle under consideration. By the law of cosines we can obtain the angles $\alpha_1$, $\beta_1$, and $\gamma_1$. 

%%----------------------------------------------------------------------------------------
%%	FIGURE 8
%%----------------------------------------------------------------------------------------
\begin{figure}
\begin{subfigure}{.5\textwidth}
  \centering
  \includegraphics[width=1\linewidth]{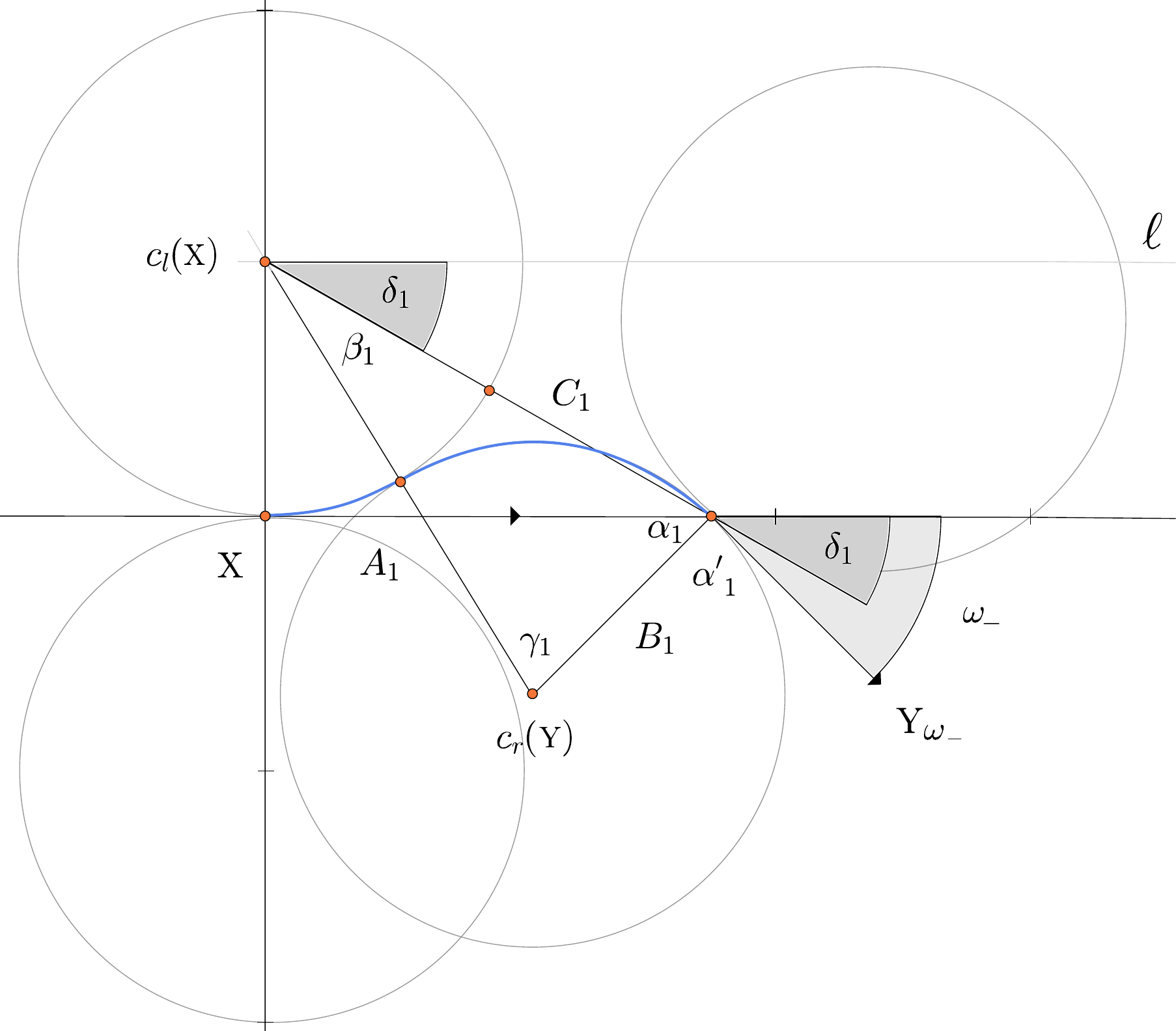}
 % \caption{}
  \label{fig:sfig1}
\end{subfigure}%
\begin{subfigure}{.5\textwidth}
  \centering
  \includegraphics[width=.87\linewidth]{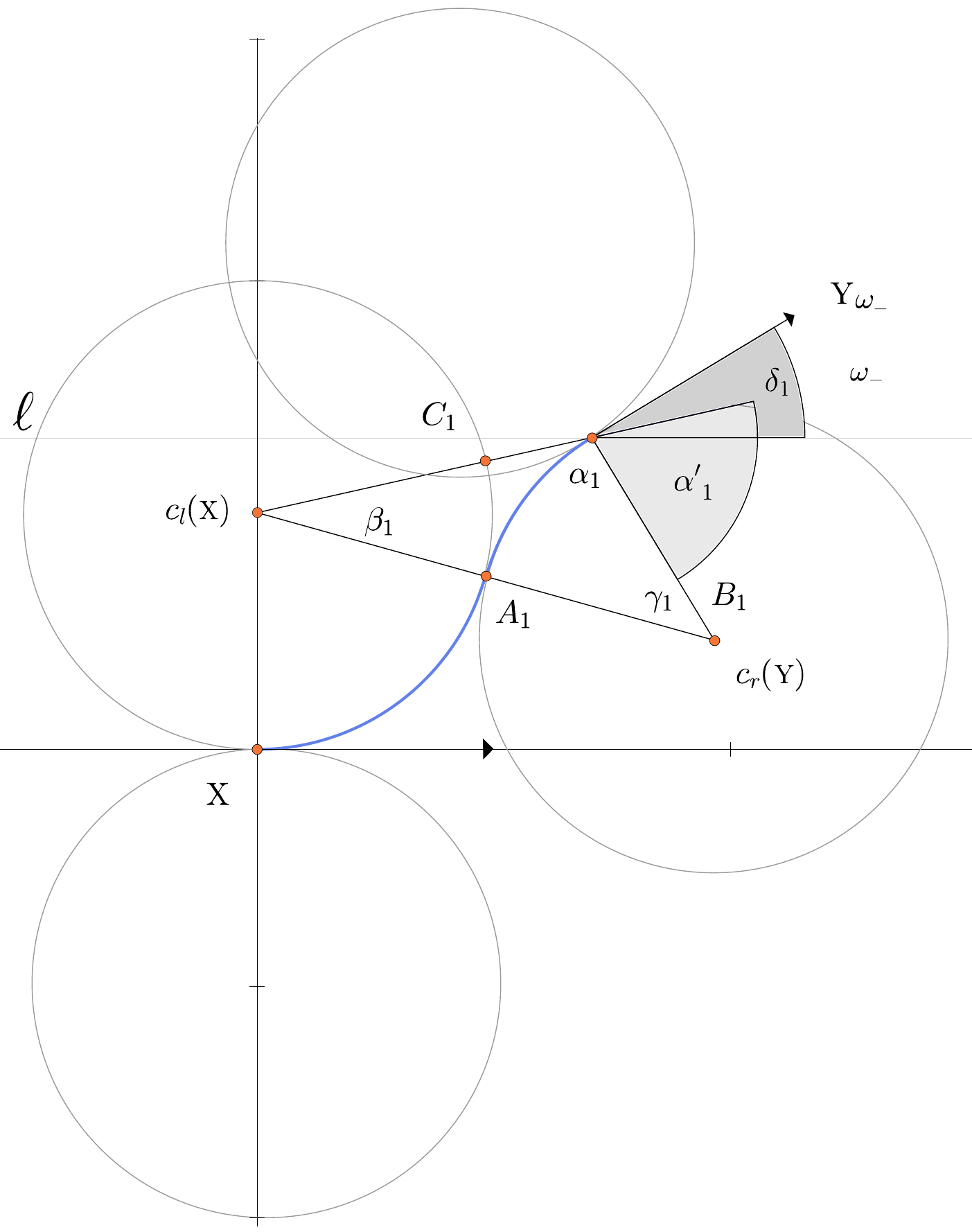}
%  \caption{}
  \label{fig:sfig2}
\end{subfigure}
\caption{Notation for short triangles. Right: The angular formulae applies when $\delta_1$ and $\omega_{-}$ (or $\omega_{+}$) have same and opposite sign. Here $\ell$ is a line parallel to the horizontal axis.}
\label{shortriang1B}
\end{figure}

Let $\delta_{1}$ be the smallest angle made by the line joining $c_l(\mbox{\sc x})$ to $y$ and the $u$-axis. The angle $\delta_1=\arctan\big(\frac{v+1}{u}\big)$ is easy to obtain since $y=(u,v)\in \mathbb R^2$ is given. 

\begin{remark}(Ruling out indeterminancies).\label{indet} Note that the standard arctan function allow us to compute angles in $(-\frac{\pi}{2}, \frac{\pi}{2})$. Since our computations involve angles in $(-\pi,\pi)$ we make use of the arctan2 function that allows us to calculate the arctangent in all four quadrants, see equations (\ref{eq:delta_a}) and (\ref{eq:delta_xa}).
\end{remark}

Depending on the final position, the angles $\delta_1$ and $\omega_-$ may have the same or different sign. In Fig. \ref{shortriang1B} left we illustrate the case when $\delta_1<0$ and $\omega_-<0$. In Fig. \ref{shortriang1B} right we illustrate the case where $\delta_1<0$ and $\omega_->0$.  Since $\alpha_{1}$ and $\alpha'_{1}$ are supplementary we have that $\alpha'_{1} = \pi - \alpha_{1}$. 

From the previous analysis we obtain that $\omega_- = \delta_{1}-\alpha'_{1}+\frac{\pi}{2}$ or equivalently:
\begin{equation} 
	\label{eq:thetar_b}	
	\omega_- = \delta_{1}+\alpha_{1}-\frac{\pi}{2}.
\end{equation} 

%----------------------------------------------------------------------------------------
%	FIGURE 9
%----------------------------------------------------------------------------------------
\begin{figure}[h]
	\centering
	\includegraphics[width=.57\textwidth,angle=0]{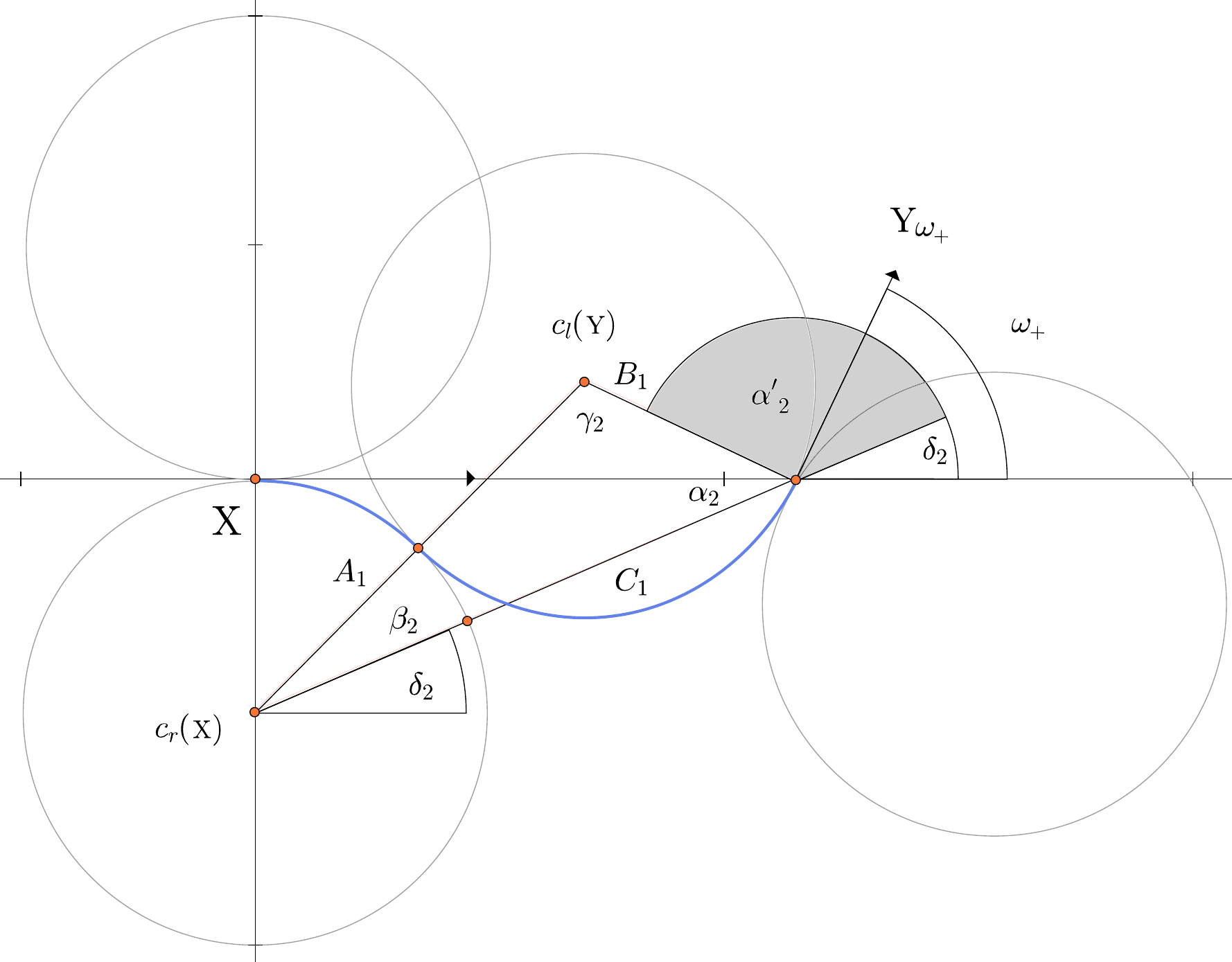}
	\caption{A critical configuration for short triangles. Note that $\delta_2>0$ and $\omega_+>0$.}
	\label{shortriang2}
\end{figure}

Now we obtain a formula for $\omega_+$, see Fig. \ref{shortriang2}. Since $\alpha_{2}$ and $\alpha'_{2}$ are supplementary we have that $\alpha'_{2} = \pi - \alpha_{2}$. 

We obtain that $\omega_+ = \delta_{2}+\alpha'_{2}-\frac{\pi}{2}$ or equivalently:

\begin{equation} 
	\label{eq:thetal_a}
	\omega_{+} = \delta_{2}-\alpha_{2}+\frac{\pi}{2}.
\end{equation} 

Here $\delta_{2}$ is the smaller angle made by the $u$-axis and the line joining $c_r(\mbox{\sc x})$ to $y$. It is easy to obtain the length of the side $C_2$.

%----------------------------------------------------------------------------------------
%	FIGURE 10
%----------------------------------------------------------------------------------------
\begin{figure}[h]
	\centering
	\includegraphics[width=.7\textwidth,angle=0]{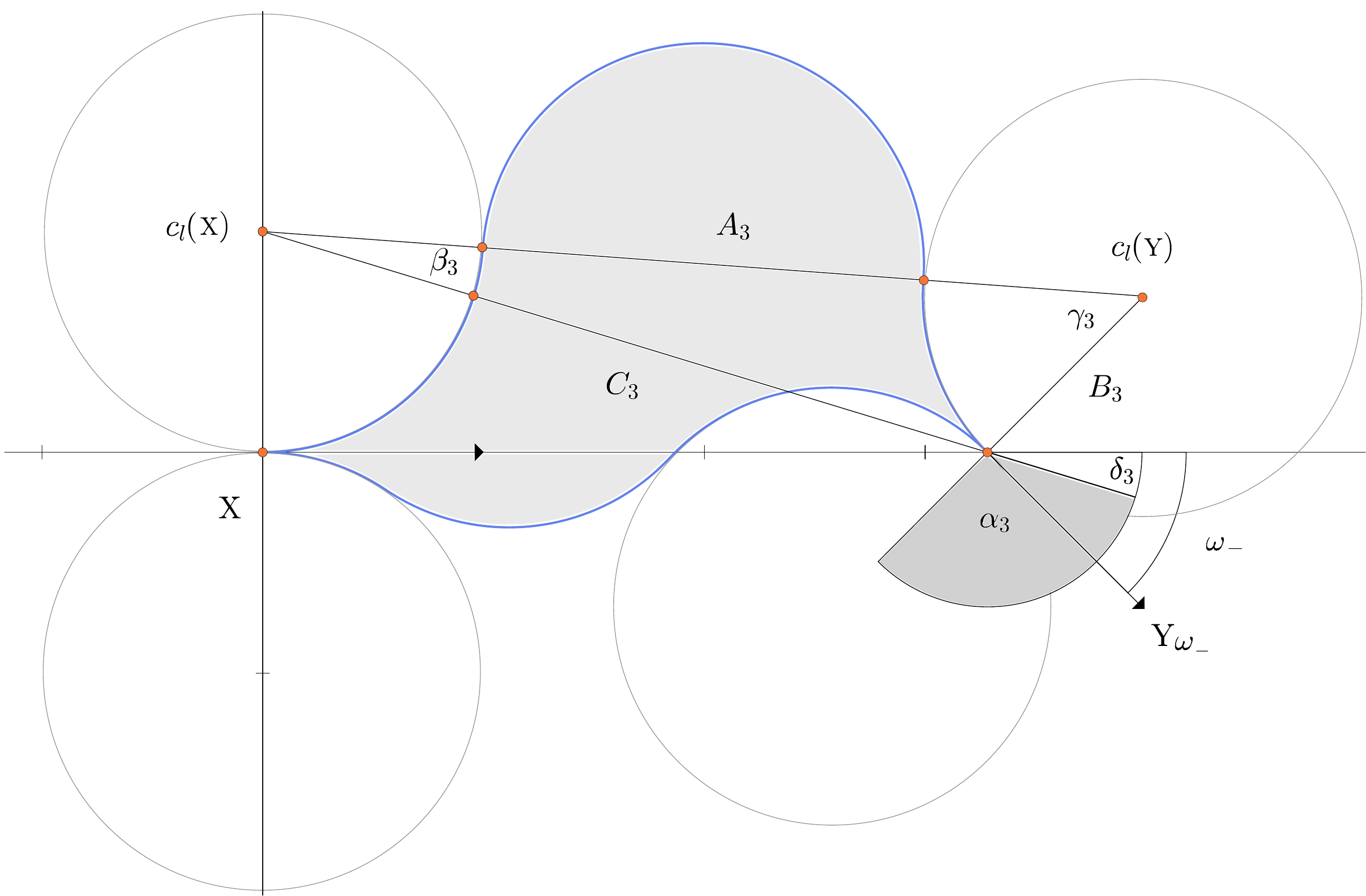}
	\caption{Notation for long triangles.}
	\label{fig:Triangle3}
\end{figure}

%----------------------------------------------------------------------------------------
%	ANGULAR FORMULAE FOR LONG TRIANGLES
%----------------------------------------------------------------------------------------
\subsection{Angular formulae for long triangles}
 Using the notation from Fig. \ref{fig:Triangle3} we present the following formulae:
\begin{equation} 
	\label{eq:thetar_c}	
	\omega_{-} = \delta_{3}-\alpha_{3}+\frac{\pi}{2}.
\end{equation} 
Similarly we obtain,
\begin{equation} 
	\label{eq:thetal_b}	
	\omega_{+} = \delta_{4}+\alpha_{4}-\frac{\pi}{2}.
\end{equation} 

%----------------------------------------------------------------------------------------
%	FIGURE 11
%----------------------------------------------------------------------------------------
\begin{figure}[h]
	\centering
	\includegraphics[width=.7\textwidth,angle=0]{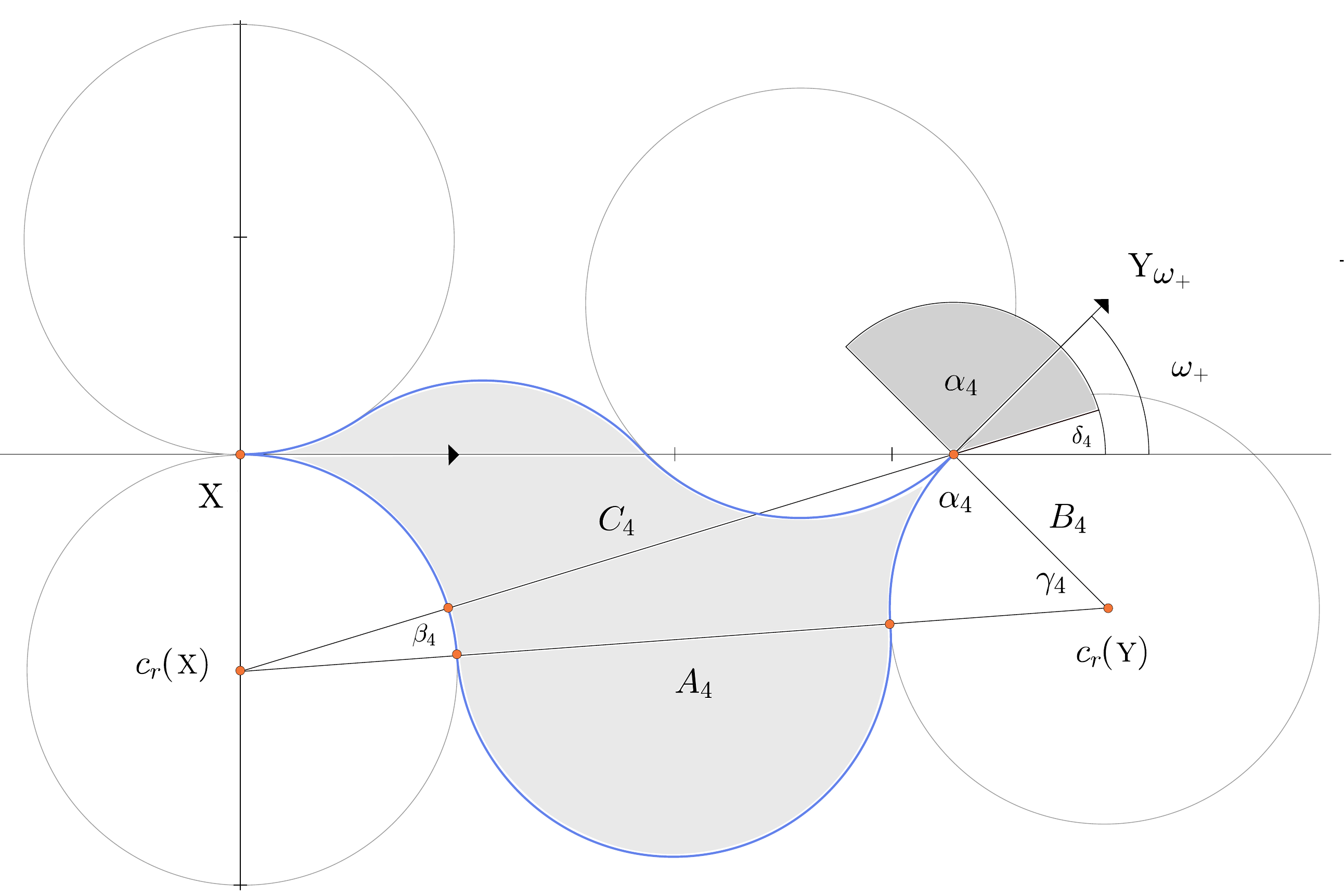}
	\caption{Notation for long triangles.}
	\label{fig:Triangle4}
\end{figure}

We put together equations (\ref{eq:thetar_b})-(\ref{eq:thetal_b}) to give explicit formulae for $\omega_-$ and $\omega_+$. 
\begin{equation} 
	\label{eq:wmin}
	\omega_{-} =
	\begin{cases}
		\delta_{1} + \alpha_{1} - \frac{\pi}{2}  & \text{if } \,\,\,d(c_l(\mbox{\sc x}),y) < 3 \\
		\delta_{3} - \alpha_{3} + \frac{\pi}{2}  & \text{if } \,\,\,d(c_l(\mbox{\sc x}),y) \geq 3
	\end{cases}
\end{equation} 
\begin{equation} 
	\label{eq:wmax}
	\omega_{+} =
	\begin{cases}
		\delta_{2} - \alpha_{2} + \frac{\pi}{2}  & \text{if } \,\,\,d(c_r(\mbox{\sc x}),y) < 3 \\
		\delta_{4} + \alpha_{4} - \frac{\pi}{2}  & \text{if } \,\,\,d(c_r(\mbox{\sc x}),y) \geq 3
	\end{cases}
\end{equation} 

%\vspace{-.2cm}
%We combine equations (\ref{eq:thetar_b})-(\ref{eq:thetal_b}) following our discussion on short and long triangles to obtain (\ref{eq:wmin}) and (\ref{eq:wmax}). Then we defined the class range function for each point $y\in \mathbb R^2$ according to their location. 
	
%----------------------------------------------------------------------------------------
%	THE CLASS RANGE 
%----------------------------------------------------------------------------------------
\section{The class range} \label{classrangesec}

Choose $(x,X), (y,Y) \in T{\mathbb R}^2$ so that the origin is identified with $x$ as in Remark \ref{coo}. We consider the formulae (\ref{eq:wmin}) and (\ref{eq:wmax}) as starting point for obtaining the class range function. The class value gives the range that $\theta$ can continuously vary so that the spaces $\Gamma(\mbox{\sc x}, \mbox{\sc y}_{\theta})$ have the same types of connected components. The class value corresponds to the length of the maximal subinterval $I(y)\subset (-\pi,\pi)$ so that there is a bounded isotopy class $\Delta(\mbox{\sc x}, \mbox{\sc y}_\theta)$, $\theta \in I(y)$. 

Next we express the angles $\delta_i$ and $\alpha_i$ in (\ref{eq:wmin}) and (\ref{eq:wmax}) in terms of generic $y=(u,v)\in \mathbb R^2$. Observe that $\delta_1=\delta_3$ since both are the acute angles made by the line joining $c_l(\mbox{x})$ and $y$ with the $u$-axis. In addition, $\delta_2=\delta_4$ since both are the acute angles made by the line joining $c_r(\mbox{x})$ and $y$ with the $u$-axis. Note that $\tan (\delta_1)=\frac{v+r}{u}$, and $\tan (\delta_2)=\frac{v-r}{u}$. 

We use arctan2 function in (\ref{eq:delta_a}) and (\ref{eq:delta_xa}) instead of the standard arctan function to determine angles $\delta_1$ and $\delta_2$, see Remark \ref{indet}. We obtain the following formulae:

\begin{equation} 
	\label{eq:delta_a}
\delta_{1}(u,v) =
		\begin{cases}
			\arctan \frac{v+1}{u}  & \text{if } u > 0 \\
			\arctan \frac{v+1}{u} + \pi  & \text{if } u < 0 \,\, \text{and}\,\, v \geq -1 \\
			\arctan \frac{v+1}{u}  - \pi & \text{if } u > 0 \,\, \text{and}\,\, v<-1\\
		\frac{\pi}{2}  & \text{if } u = 0  \,\, \text{and}\,\, v>-1\\
			-\frac{\pi}{2}  & \text{if } u = 0  \,\, \text{and}\,\, v<-1\\
		\end{cases}
		\end{equation} 
\begin{equation} 
	\label{eq:delta_xa}
\delta_{2}(u,v) = 
		\begin{cases}
			\arctan \frac{v-1}{u}  & \text{if } u > 0 \\
			\arctan \frac{v-1}{u} + \pi  & \text{if } u < 0 \,\, \text{and}\,\, v \geq 1 \\
			\arctan \frac{v-1}{u}  - \pi & \text{if } u > 0 \,\, \text{and}\,\, v<1\\
			 \frac{\pi}{2}  & \text{if } u = 0  \,\, \text{and}\,\, v>1\\
			-\frac{\pi}{2}  & \text{if } u = 0  \,\, \text{and}\,\, v<1\\
		\end{cases}
\end{equation} 

Now we determine the angles $\alpha_{i}$ in  (\ref{eq:wmin}) and (\ref{eq:wmax}). To this end, we apply the law of cosines. Here we are not considering degenerate triangles, so the following formulae are never undetermined. 
\begin{equation*} 
	\alpha_{i} = \arccos\bigg(\frac{b_i^2+c_i^2-a_i^2}{2b_ic_i}\bigg)
\end{equation*} 

 Since the initial and final positions are given, the coordinates of the adjacent circles are easily obtained. In consequence, the length of the sides of the short and long triangles are easily obtained. 
 
 We can express the angles $\alpha_{i}$ as a function of $y=(u,v)\in \mathbb R^2$. That is:
\begin{equation}
	\label{eq:alpha_1} 
	\alpha_{1}(u,v)  = \arccos\bigg(\frac{u^2+(-1-v)^2-3}{2\sqrt{u^2+(-1-v)^2}}\bigg)
\end{equation} 
\vspace{0em}
\begin{equation} 
	\label{eq:alpha_2} 
	\alpha_{2}(u,v)  = \arccos\bigg(\frac{u^2+(1-v)^2-3}{2\sqrt{u^2+(1-v)^2}}\bigg)
\end{equation} 
\vspace{0em}
\begin{equation} 
	\label{eq:alpha_3} 
	\alpha_{3}(u,v)  = \arccos\bigg(\frac{u^2+(-1-v)^2-15}{2\sqrt{u^2+(-1-v)^2}}\bigg)
\end{equation} 
\vspace{0em}
\begin{equation} 
	\label{eq:alpha_4} 
	\alpha_{4}(u,v)  = \arccos\bigg(\frac{u^2+(1-v)^2-15}{2\sqrt{u^2+(1-v)^2}}\bigg)
\end{equation} 

Note that we have expressed all the angles $\alpha_i$ and $\delta_i$ as functions of the variables $u$ and $v$. In addition, recall that in Definition \ref{defw} we considered the concept of critical angles $\omega_-$ and $\omega_+$. We abuse notation and define the functions $\omega_-:\mathbb R^2\to\mathbb R$ and $\omega_+:\mathbb R^2\to\mathbb R$. They have been constructed to match Definition \ref{defw}. These functions assign to each final position $y=(u,v)\in \mathbb R^2$ its respective critical angle $\omega_-(y)$ and $\omega_-(y)$. 

We consider equations (\ref{eq:delta_a})-(\ref{eq:alpha_4}) according to equations (\ref{eq:wmin}) and (\ref{eq:wmax}) to obtain:

\begin{equation}
	\label{eq:Wmin} 
	\omega_-(u,v) =
	\begin{cases}
		\begin{cases}
				\arctan (\frac{v-1}{u}) + \arccos(\frac{(u^2+(1-v)^2)-3}{2\sqrt{u^2+(1-v)^2}}) - \frac{\pi}{2}  & \text{if } u > 0 \\
				\begin{cases}
					\arctan (\frac{v-1}{u}) + \pi + \arccos(\frac{(u^2+(1-v)^2)-3}{2\sqrt{u^2+(1-v)^2}}) - \frac{\pi}{2}  & \text{if } v \geq 1 \\
					\arctan (\frac{v-1}{u}) - \pi + \arccos(\frac{(u^2+(1-v)^2)-3}{2\sqrt{u^2+(1-v)^2}}) - \frac{\pi}{2}  & \text{if } v < 1
				\end{cases} & \text{if } u < 0 \\
				\begin{cases}
					\frac{\pi}{2} + \arccos(\frac{(u^2+(1-v)^2)-3}{2\sqrt{u^2+(1-v)^2}}) - \frac{\pi}{2}  & \text{if } v>1\\
					-\frac{\pi}{2} + \arccos(\frac{(u^2+(1-v)^2)-3}{2\sqrt{u^2+(1-v)^2}}) - \frac{\pi}{2}  & \text{if } v<1\\
				\end{cases} & \text{if } u = 0 \\
			\end{cases} & \text{if } d(\mbox{\it c}_l(\mbox{\sc x}), y) < 3 \\
		\begin{cases}
				\arctan (\frac{v-1}{u}) - \arccos(\frac{(u^2+(1-v)^2)-15}{2\sqrt{u^2+(1-v)^2}}) + \frac{\pi}{2} & \text{if } u > 0 \\
				\begin{cases}
					\arctan (\frac{v-1}{u}) + \pi - \arccos(\frac{(u^2+(1-v)^2)-15}{2\sqrt{u^2+(1-v)^2}}) + \frac{\pi}{2}  & \text{if } v \geq 1 \\
					\arctan (\frac{v-1}{u}) - \pi - \arccos(\frac{(u^2+(1-v)^2)-15}{2\sqrt{u^2+(1-v)^2}}) + \frac{\pi}{2}  & \text{if } v < 1
				\end{cases} & \text{if } u < 0 \\
				\begin{cases}
					\frac{\pi}{2} - \arccos(\frac{(u^2+(1-v)^2)-15}{2\sqrt{u^2+(1-v)^2}}) + \frac{\pi}{2}   & \text{if } v>1\\
					-\frac{\pi}{2} - \arccos(\frac{(u^2+(1-v)^2)-15}{2\sqrt{u^2+(1-v)^2}}) + \frac{\pi}{2}   & \text{if } v<1\\
				\end{cases} & \text{if } u = 0 \\
			\end{cases} & \text{if } d(\mbox{\it c}_l(\mbox{\sc x}), y) \geq 3
	\end{cases} 
\end{equation} 
\vspace{0em}

\begin{equation} 
	\label{eq:Wmax} 
	\omega_+(u,v) =
	\begin{cases}
		\begin{cases}
				\arctan (\frac{v+1}{u}) - \arccos(\frac{(u^2+(-1-v)^2)-3}{2\sqrt{u^2+(-1-v)^2}}) + \frac{\pi}{2} & \text{if } u > 0 \\
				\begin{cases}
					\arctan (\frac{v+1}{u}) + \pi - \arccos(\frac{(u^2+(-1-v)^2)-3}{2\sqrt{u^2+(-1-v)^2}}) + \frac{\pi}{2} & \text{if } v \geq -1 \\
					\arctan (\frac{v+1}{u}) - \pi - \arccos(\frac{(u^2+(-1-v)^2)-3}{2\sqrt{u^2+(-1-v)^2}}) + \frac{\pi}{2} & \text{if } v < -1
				\end{cases} & \text{if } u < 0\\
				\begin{cases}
					\frac{\pi}{2}   - \arccos(\frac{(u^2+(-1-v)^2)-3}{2\sqrt{u^2+(-1-v)^2}}) + \frac{\pi}{2}  & \text{if } v>-1\\
					-\frac{\pi}{2}  - \arccos(\frac{(u^2+(-1-v)^2)-3}{2\sqrt{u^2+(-1-v)^2}}) + \frac{\pi}{2}  & \text{if } v<-1\\
				\end{cases} & \text{if } u = 0 \\
			\end{cases}  & \text{if } d(\mbox{\it c}_r(\mbox{\sc x}), y) < 3 \\
		\begin{cases}
				\arctan (\frac{v+1}{u}) + \arccos(\frac{(u^2+(-1-v)^2)-15}{2\sqrt{u^2+(-1-v)^2}}) - \frac{\pi}{2}  & \text{if } u > 0 \\
				\begin{cases}
					\arctan (\frac{v+1}{u}) + \pi + \arccos(\frac{(u^2+(-1-v)^2)-15}{2\sqrt{u^2+(-1-v)^2}}) - \frac{\pi}{2}  & \text{if } v \geq -1\\
					\arctan (\frac{v+1}{u}) - \pi + \arccos(\frac{(u^2+(-1-v)^2)-15}{2\sqrt{u^2+(-1-v)^2}}) - \frac{\pi}{2}  & \text{if } v < -1
				\end{cases} & \text{if } u < 0\\
				\begin{cases}
					\frac{\pi}{2}  - \arccos(\frac{(u^2+(-1-v)^2)-15}{2\sqrt{u^2+(-1-v)^2}}) + \frac{\pi}{2}   & \text{if } v>-1\\
					-\frac{\pi}{2}  - \arccos(\frac{(u^2+(-1-v)^2)-15}{2\sqrt{u^2+(-1-v)^2}}) + \frac{\pi}{2}  & \text{if } v<-1\\
				\end{cases} & \text{if } u = 0 \\
			\end{cases} & \text{if } d(\mbox{\it c}_r(\mbox{\sc x}), y) \geq 3
	\end{cases}
\end{equation} 

The undetermined expression $\frac{u}{v}$ for $u=v=0$ can only happen at the centres of $C_l(\mbox{\sc x})$ and $C_r(\mbox{\sc x})$. By Corollary 3.4 in \cite{papere}, no path in a bounded isotopy class can satisfy the points $(0,1)$ or $(0,-1)$. This is due the existence of parallel tangents; contradicting Theorem 7.12 in \cite{paperc}. 

\begin{definition} \label{classrange+} The class range function $\Theta: \mathbb R^2 \to \mathbb R$ is defined to be: 
\label{eq:angularrange_1}	
	$$\Theta(y)= \omega_+(y) - \omega_-(y)\geq 0.$$
\end{definition}

It is important to note that the critical angles $\omega_-(y)$ and $\omega_+(y)$ are chosen so that $\omega_+(y)\geq\omega_-(y)$, see Definition \ref{defw}. In addition, note that $\omega_-(y)$ and $\omega_+(y)$ are in the boundary or the closure of the interval $I(y)$. Of course, if $\omega_+(y) - \omega_-(y)<0$ we have that $I(y)=\emptyset$, and so $\Omega=\emptyset$, and so there is no bounded isotopy class $\Delta(\mbox{\sc x,y})$. The case where $\Theta(y)=0$ is discussed bellow. 

The interior, closure, boundary and complement of a set $B$ are denoted by $int(B)$, $cl(B)$, $\partial(B)$, and $B^c$ respectively. 

Next we present data obtained after plotting the values of $\omega_+(y) - \omega_-(y)$. 

\subsection{Facts about the class range function}\label{rem:data} (see Fig. \ref{figraf}).

\begin{itemize} \label{facts}
\item $\Theta$ is continuous.
\item The domain of $\Theta$ is a bounded set $B\subset \mathbb R^2$. In Section \ref{classdomain} we determine $B$ and its subdivisions. In these subdivisions lie the final positions $y\in \mathbb R^2$ so that $\Gamma(\mbox{\sc x},\mbox{\sc y}_\theta)$ are fibers of the same type, $\mbox{\sc y}_\theta=(y,Y_\theta)$, see Definition \ref{def:spaces}.
\item If  $y\in int(B)$, then $\Theta(y)>0$. 
\item If $y \in \partial (cl(B))$, then $\Theta(y)=0$.
\item If $y\in B^c$, then $\omega_+(y) - \omega_-(y)<0$. In this case, $I(y)=\emptyset$, and so $\Omega=\emptyset$ or equivalently there is no bounded isotopy class $\Delta(\mbox{\sc x,y})$.
\item  The range of $\Theta$ is the interval $[0, \arctan \big(\frac{1}{4}\sqrt{2}\big)+\pi]$.
\item $\Theta$ attains the minima at the final positions $y\in \mathbb R^2$ for  {\sc c}  (or {\sc cc}) isolated points. Here we have that $\Theta(y)=0$.
\item $\Theta$ attains a maximum at $y=(0,2\sqrt{2})$ with $\Theta(y)= \arctan \big(\frac{1}{4}\sqrt{2}\big)+\pi$. In Fig. \ref{regparam} top we illustrate the class range for $y=(0,2\sqrt{2})$.
\end{itemize}

%----------------------------------------------------------------------------------------
%	FIGURE 12
%----------------------------------------------------------------------------------------
\begin{figure}[h]
	\centering
	\includegraphics[width=1\textwidth,angle=0]{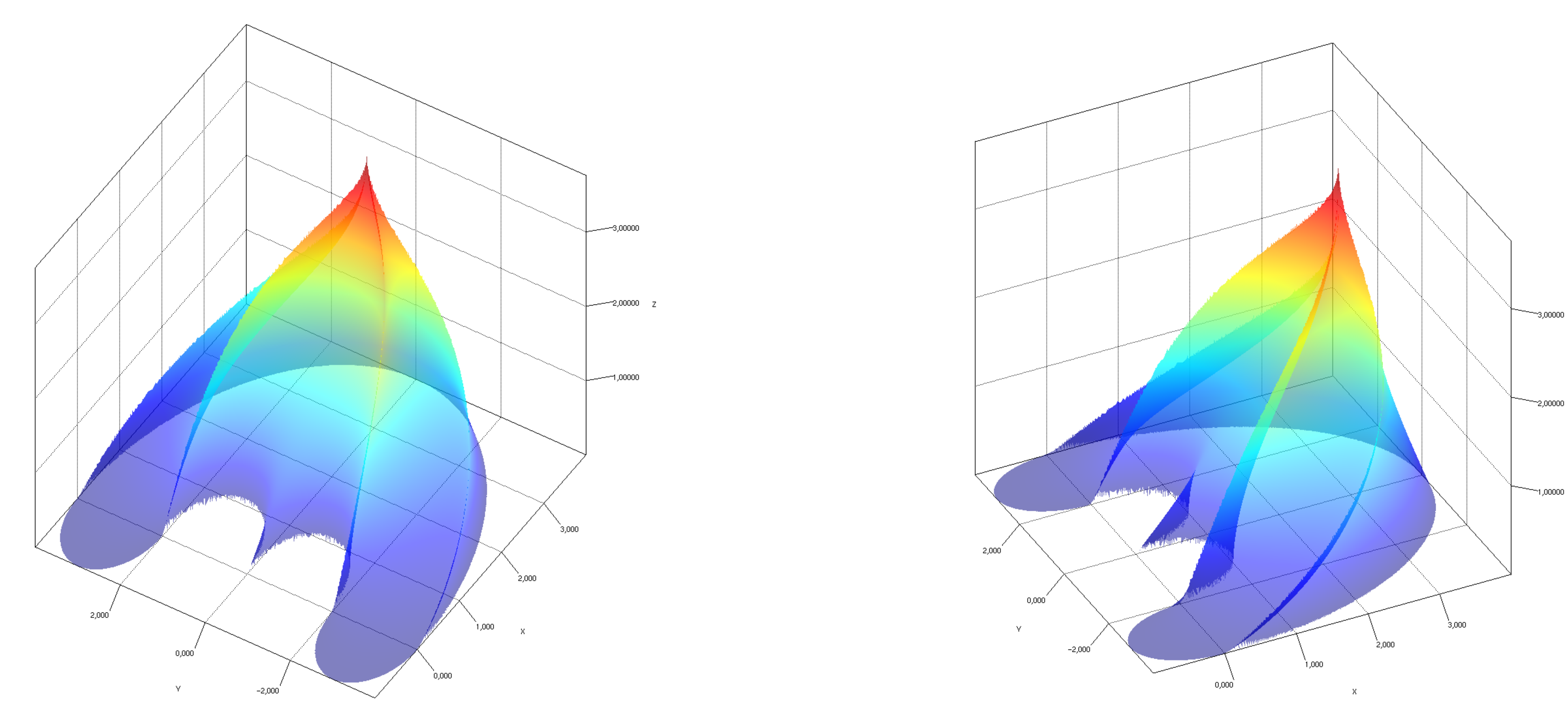}
	\caption{The graph of the class range function, see Definition \ref{classrange+}. Note that the class range is constructed out of (\ref{eq:Wmin}) and (\ref{eq:Wmax}) and that these functions are obtained by combining (\ref{eq:wmin}) and (\ref{eq:wmax}).}
  \label{figraf}
\end{figure}

%----------------------------------------------------------------------------------------
%	FIGURE 12
%----------------------------------------------------------------------------------------
%\begin{figure}
%\begin{subfigure}{.5\textwidth}
%  \centering
%  \includegraphics[width=1\linewidth]{image1}
%% \caption{}
%  \label{figraf}
%\end{subfigure}%
%\begin{subfigure}{.5\textwidth}
%  \centering
%  \includegraphics[width=1\linewidth]{image2}
%%  \caption{}
%%  \label{fig:sfig1}
%\end{subfigure}
%\caption{The graph of the class range function, see Definition \ref{classrange+}. Note that the class range is constructed out of (\ref{eq:Wmin}) and (\ref{eq:Wmax}) and that these functions are obtained by combining (\ref{eq:wmin}) and (\ref{eq:wmax}).}
%  \label{figraf}
%\end{figure}

%----------------------------------------------------------------------------------------
%	CLASS DOMAIN
%----------------------------------------------------------------------------------------
\section{Class domain}\label{classdomain}

The obvious observation that a triangle with sides of length $1$ and $2$ cannot have a third side of length greater to $3$ leads us to study the cases: 
\begin{equation}\label{ineq:1} d(c_l(\mbox{\sc x}),y) < 3
\end{equation}
\begin{equation}\label{ineq:2} d(c_l(\mbox{\sc x}),y) \geq 3
\end{equation}
\begin{equation}\label{ineq:3} d(c_r(\mbox{\sc x}),y) < 3
\end{equation}
\begin{equation}\label{ineq:4} d(c_r(\mbox{\sc x}),y) \geq 3
\end{equation}
 
After looking at the four possible combinations for: 
 \begin{equation} \label{eq:0}
 \omega_-=\omega_+
\end{equation}
in equations (\ref{eq:wmin}) and (\ref{eq:wmax}) we obtain:
  \begin{equation} \label{eq:1}
	\delta_{1} + \alpha_{1} - \frac{\pi}{2} = \delta_{2} - \alpha_{2} + \frac{\pi}{2}
\end{equation}
\vspace{-.7em}
\begin{equation} \label{eq:2}
	\delta_{3} - \alpha_{3} + \frac{\pi}{2} = \delta_{2} - \alpha_{2} + \frac{\pi}{2}
\end{equation}
\vspace{-.7em}
\begin{equation}\label{eq:3}
	\delta_{1} + \alpha_{1} - \frac{\pi}{2} = \delta_{4} + \alpha_{4} - \frac{\pi}{2}
\end{equation} 
\begin{equation}\label{eq:4}
	\delta_{3} - \alpha_{3} + \frac{\pi}{2} = \delta_{4} + \alpha_{4} - \frac{\pi}{2}
\end{equation} 

The following observations regarding the circles (\ref{eq:circ_a})-(\ref{eq:circ_g}) can be checked by a mere evaluation. We leave the details to the reader. 

%Recall that $y$ is a vertex for short and long triangles. 

Consider $\mbox{\sc x}\in T\mathbb R^2$ according to Remark \ref{coo}. The locus of the circle (\ref{eq:circ_a}) is satisfied by the final positions $y=(u,v)\in \mathbb R^2$ with $u\geq0$. In this case, the angles in the associated triangles (according to Section \ref{trans}) satisfy (\ref{eq:4}), see Figs. \ref{fig:Triangle3}-\ref{fig:Triangle4}. In addition, $\Theta(u,v)=0$ for points in (\ref{eq:circ_a}) for $u\geq0$. 
%It is not hard to see that no region $\Omega$ can be associated with these final positions.
%For each of the points in (\ref{eq:circ_a}) with $u\geq0$ we can consider a final direction so that
\begin{equation} 
	\label{eq:circ_a}
		u^2+v^2=16.
\end{equation}  
The loci of the circles (\ref{eq:circ_d}) and (\ref{eq:circ_e}) for $u\geq 0$ are satisfied by the final positions $y=(u,v)\in \mathbb R^2$ so that its associated angles according to subsection \ref{short} satisfy (\ref{eq:1}). In addition, $\Theta(y)=0$ for points in (\ref{eq:circ_d}) and (\ref{eq:circ_e}) with $u\geq0$. 
\begin{equation}
	\label{eq:circ_d}
		u^2+(v-1)^2=1
\end{equation}
\begin{equation}
	\label{eq:circ_e}
		u^2+(v+1)^2=1
\end{equation}

The locus of the circle (\ref{eq:circ_f}) is satisfied by the final positions $y=(u,v)\in \mathbb R^2$ with for $u\leq 0$ so that its associated angles according to subsections \ref{short} and \ref{long} satisfy (\ref{eq:2}).

 The locus of the circle (\ref{eq:circ_g}) is satisfied by the final positions $y=(u,v)\in \mathbb R^2$ with for $u\leq 0$ so that its associated angles according to subsections \ref{short} and \ref{long} satisfy (\ref{eq:3}). In addition, $\Theta(y)=0$ for points in (\ref{eq:circ_f}) and (\ref{eq:circ_g}) with $u\leq0$. 
\begin{equation} 
	\label{eq:circ_f}
		u^2+(v-3)^2=1
\end{equation} 
\begin{equation} 
	\label{eq:circ_g}
		u^2+(v+3)^2=1
\end{equation} 
The circles (\ref{eq:circ_b}) and (\ref{eq:circ_c}) are trivially extracted out of relations (\ref{ineq:1})-(\ref{ineq:4})
\begin{equation} 
	\label{eq:circ_b}
		u^2+(v-1)^2=9
\end{equation} 
\vspace{-2em}
\begin{equation} 
	\label{eq:circ_c}
		u^2+(v+1)^2=9
\end{equation} 

%----------------------------------------------------------------------------------------
%	FIGURE 13
%----------------------------------------------------------------------------------------
\begin{figure}[h]
	\centering
	\includegraphics[width=1\textwidth,angle=0]{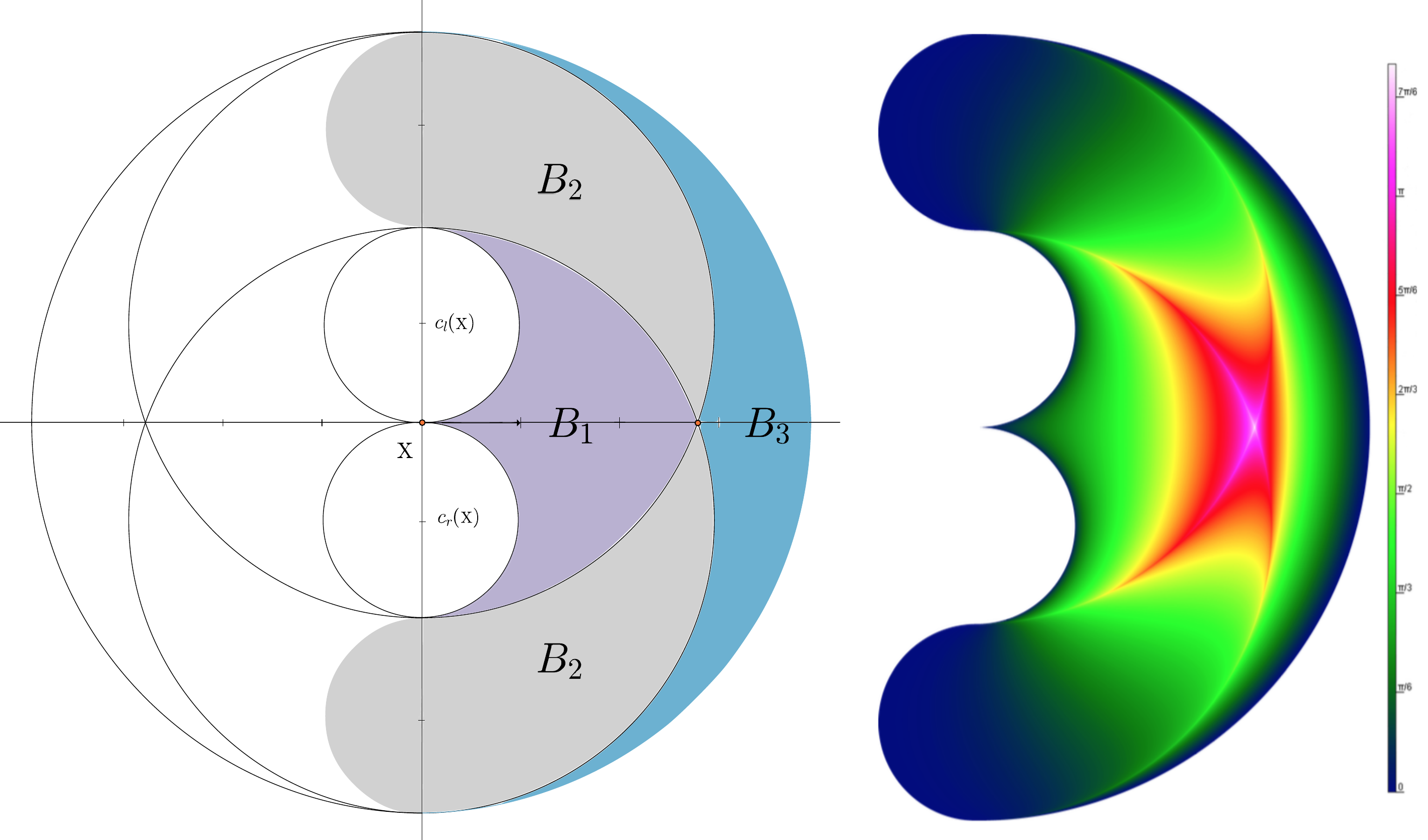}
	\caption{Left: The domain $B\subset \mathbb R^2$ of $\Theta$. Note that $B$ is bounded by the circles (\ref{eq:circ_a})-(\ref{eq:circ_c}). Right: The temperature give the length of $I(y)$ for $y\in B$. }
	\label{figRegionB}
\end{figure}
%The darkest blue show range class function values close to 0.}

%----------------------------------------------------------------------------------------
%	DESCRIPTION OF B
%----------------------------------------------------------------------------------------
\subsection{Description of $B\subset \mathbb R^2$}\label{descB}
We obtain the domain of the class range function by evaluating (\ref{eq:Wmin}) and (\ref{eq:Wmax}) according to Definition \ref{classrange+}, this planar set is represented by the colored portion in Fig. \ref{figRegionB} left. In Fig. \ref{figRegionB} right we show a Heatmap for the class values.

\begin{definition}The domain $B\subset \mathbb R^2$ of the class range function $\Theta:\mathbb R^2\to\mathbb R$ corresponds to the open bounded portion enclosed by the simple closed curve corresponding to the union of the semicircles (\ref{eq:circ_d}) and (\ref{eq:circ_e}) for $u\geq0$; (\ref{eq:circ_f}) and (\ref{eq:circ_g}) for $u\leq0$; and (\ref{eq:circ_a}) for $u\geq0$, union the semicircles (\ref{eq:circ_d}) and (\ref{eq:circ_e}) for $u>0$, union the origin.
\end{definition}

Observe that the semicircles (\ref{eq:circ_d}) and (\ref{eq:circ_e}) for $u>0$ are the location where {\sc c} isolated points are defined. In addition, $\Theta$ is continuous but not differentiable at $B$ intersection the circles (\ref{eq:circ_f}) and (\ref{eq:circ_g}), see the list of facts in \ref{facts}. 

\begin{definition} \label{cbc} Let $\mbox{\sc x,y}\in T\mathbb R^2$. Then

\begin{itemize}
		\item $y\in B_{1}\subset{B}$ then, $d(\mbox{\sc c}_l(\mbox{\sc x}), y) < 3$ and $d(\mbox{\sc c}_r(\mbox{\sc x}), y) < 3$ are satisfied.
		\item $y\in B_{2}\subset{B}$ then, $d(\mbox{\sc c}_l(\mbox{\sc x}), y) < 3$ and $d(\mbox{\sc c}_r(\mbox{\sc x}), y) \geq 3$ or, \\
		$d(\mbox{\sc c}_l(\mbox{\sc x}), y) \geq 3$ and $d(\mbox{\sc c}_r(\mbox{\sc x}), y) < 3$ are satisfied.
		\item $y\in B_{3}\subset{B}$ then, $d(\mbox{\sc c}_l(\mbox{\sc x}), y) \geq 3$ and $d(\mbox{\sc c}_r(\mbox{\sc x}), y) \geq 3$ are satisfied.
		\item set $B_4=B^c$.
	\end{itemize}

\end{definition}
 
%----------------------------------------------------------------------------------------
%	MAIN EXISTENCE 1-PARAMETER
%----------------------------------------------------------------------------------------
\begin{theorem}\label{existvect} Given $\mbox{\sc x} \in T\mathbb R^2$ and $y\in B$. There exists a family  $e^{i\theta}=Y_\theta\in T_y\mathbb R^2$, $\theta\in (-\pi,\pi)$, such that for  $\mbox{\sc x}, \mbox{\sc y}_\theta \in T\mathbb R^2$ we have that $\Delta(\mbox{\sc x}, \mbox{\sc y}_\theta)$ is a one-parameter family of bounded isotopy classes. 
\end{theorem}

\begin{proof} Consider $\mbox{\sc x} \in T\mathbb R^2$. Recall that the values of $\Theta$ are determined by a combination of short and long triangles, according to subsections \ref{short} and \ref{long}. 

Since $y\in B$ then $\Theta(y)\geq0$, see the facts in \ref{rem:data}. 

Suppose that $\Theta(y)>0$, then we have that $\omega_-\neq \omega_+$. This immediately implyies that $Y_{\omega_-}\neq Y_{\omega_+}$. We conclude that the bounded classes $\Delta(\mbox{\sc x}, \mbox{\sc y}_{\omega_-})\neq \Delta(\mbox{\sc x}, \mbox{\sc y}_{\omega_+})$. Since $\theta$ is continuous, by the intermediate value theorem the result follows. If $\Theta(y)=0$ then $\Delta(\mbox{\sc x}, \mbox{\sc y}_{\theta})$ is a {\sc c} isolated point with $I(y)$ being a single point. 
\end{proof}

%----------------------------------------------------------------------------------------
%  TYPES OF SPACES
%----------------------------------------------------------------------------------------
\begin{definition}\label{def:spaces}\hfill 
%----------------------------------------------------------------------------------------
%  TYPE I
%----------------------------------------------------------------------------------------
\begin{enumerate}
\item A family of spaces $\Gamma(\mbox{\sc x}, \mbox{\sc y}_\theta)$ such that:
\begin{itemize} 
\item  for $\theta=\omega_-$ we have that $\Delta(\mbox{\sc x}, \mbox{\sc y}_{\omega_-})$ is a {\sc cc} isolated point in $\Gamma(\mbox{\sc x}, \mbox{\sc y}_{\omega_-})$;
\item  for $\theta \in (\omega_-,\omega_+)$ there exists a bounded $\Delta(\mbox{\sc x}, \mbox{\sc y}_\theta)$;
\item for $\theta=\omega_+$ we have that $\Delta(\mbox{\sc x}, \mbox{\sc y}_{\omega_+})$ is a {\sc cc} isolated point in $\Gamma(\mbox{\sc x}, \mbox{\sc y}_{\omega_+})$;
\item for $\theta \notin [\omega_-,\omega_+]$ there is no bounded $\Delta(\mbox{\sc x}, \mbox{\sc y}_\theta)$
\end{itemize}
 is called a {\bf fiber of type I}.
 \vspace{.2cm}
 
% there exists a bounded non-empty $\Delta(\mbox{\sc x}, \mbox{\sc y}_\theta)$
% 
% there is no non-empty bounded $\Delta(\mbox{\sc x}, \mbox{\sc y}_\theta)$

%----------------------------------------------------------------------------------------
%  TYPE II
%----------------------------------------------------------------------------------------
\item A family of spaces $\Gamma(\mbox{\sc x}, \mbox{\sc y}_\theta)$ such that:

\begin{itemize} 
\item for $\theta=\omega_-$ we have that $\Delta(\mbox{\sc x}, \mbox{\sc y}_{\omega_-})$ is a {\sc cc} isolated point in $\Gamma(\mbox{\sc x}, \mbox{\sc y}_{\omega_-})$;

\item for $\theta \in (\omega_-,\omega_+]$ there exists a bounded $\Delta(\mbox{\sc x}, \mbox{\sc y}_\theta)$;

\item for $\theta \notin (\omega_-,\omega_+]$ there is no bounded $\Delta(\mbox{\sc x}, \mbox{\sc y}_\theta)$. 
\end{itemize}
Or,
\begin{itemize} 
\item for $\theta=\omega_+$ we have that $\Delta(\mbox{\sc x}, \mbox{\sc y}_{\omega_+})$ is a {\sc cc} isolated point in $\Gamma(\mbox{\sc x}, \mbox{\sc y}_{\omega_+})$;

\item for $\theta \in [\omega_-,\omega_+)$ there exists a bounded $\Delta(\mbox{\sc x}, \mbox{\sc y}_\theta)$;

\item for $\theta \notin [\omega_-,\omega_+)$ there is no bounded $\Delta(\mbox{\sc x}, \mbox{\sc y}_\theta)$
\end{itemize}

 is called a {\bf fiber of type II}.
 \vspace{.2cm}

%----------------------------------------------------------------------------------------
%  TYPE III
%----------------------------------------------------------------------------------------
\item A family of spaces $\Gamma(\mbox{\sc x}, \mbox{\sc y}_\theta)$ such that:

\begin{itemize}
\item for $\theta \in [\omega_-,\omega_+]$ there exists a bounded $\Delta(\mbox{\sc x}, \mbox{\sc y}_\theta)$;
\item for $\theta \notin [\omega_-,\omega_+]$ there is no bounded $\Delta(\mbox{\sc x}, \mbox{\sc y}_\theta)$
\end{itemize}
 is called a {\bf fiber of type III}.

 \vspace{.2cm}

%----------------------------------------------------------------------------------------
%  TYPE IV
%----------------------------------------------------------------------------------------
\item A family of spaces $\Gamma(\mbox{\sc x}, \mbox{\sc y}_\theta)$ such that:
\begin{itemize}
\item there is no bounded $\Delta(\mbox{\sc x}, \mbox{\sc y}_\theta)$ for all $\theta \in (-\pi,\pi]$
\end{itemize}
 is called a {\bf fiber of type IV}.
 \vspace{.2cm}

%----------------------------------------------------------------------------------------
%  TYPE V
%----------------------------------------------------------------------------------------
\item A family of spaces $\Gamma(\mbox{\sc x}, \mbox{\sc y}_\theta)$ is called a {\bf fiber of type V} if $x=y$. In this case, each $\Gamma(\mbox{\sc x}, \mbox{\sc y}_\theta)$, $\theta\in (-\pi,\pi]$ admit an isolated point, being a path of length zero. In addition, there is no bounded $\Delta(\mbox{\sc x}, \mbox{\sc y}_\theta)$ for all $\theta \in (-\pi,\pi]$.
\end{enumerate}
\end{definition}

%----------------------------------------------------------------------------------------
%	MAIN THEOREM BABY
%----------------------------------------------------------------------------------------
\begin{theorem}\label{maincensus1}Consider $\mbox{\sc x},\mbox{\sc y}_\theta \in T\mathbb R^2$ with $x\neq y$. 

Suppose $y\in B\subset \mathbb R^2$, then:
\begin{itemize}
\item If $y\in B_1$, then $\Gamma(\mbox{\sc x}, \mbox{\sc y}_\theta)$ is a fiber of type I. 
\item  If $y\in B_2$, then $\Gamma(\mbox{\sc x}, \mbox{\sc y}_\theta)$ is a fiber of type II.
\item If $y\in B_3$, then $\Gamma(\mbox{\sc x}, \mbox{\sc y}_\theta)$ is a fiber of type III.
\end{itemize}
If $y\in B_4=B^c$, then $\Gamma(\mbox{\sc x}, \mbox{\sc y}_\theta)$ is a fiber of type IV.

\noindent If $x=y$, then $\Gamma(\mbox{\sc x}, \mbox{\sc y}_\theta)$ is a fiber of type V.
\end{theorem}

\begin{proof} If $y\in B_1$, the endpoints $\mbox{\sc x},\mbox{\sc y}_\theta \in T\mathbb R^2$ have associated two short triangles, according to equations (\ref{eq:thetar_b}) and (\ref{eq:thetal_a}). Via equations (\ref{eq:Wmin}) and (\ref{eq:Wmax}) we obtain the values $\omega_-$ and $\omega_+$. Since $\Theta(y)>0$ (see the facts in \ref{rem:data}) Theorem \ref{existvect}, guarantees the existence of a family of bounded isotopy classes $\Delta(\mbox{\sc x}, \mbox{\sc y}_\theta)\subset \Gamma(\mbox{\sc x}, \mbox{\sc y}_\theta)$ for $ \theta \in (\omega_-,\omega_+)=I(y)$. Note that by construction $\Gamma(\mbox{\sc x}, \mbox{\sc y}_{\omega_-})$ and $\Gamma(\mbox{\sc x}, \mbox{\sc y}_{\omega_+})$ admit a {\sc cc} isolated point each. 

If $y\in B_2$, then the endpoints $\mbox{\sc x},\mbox{\sc y}_\theta \in T\mathbb R^2$ have associated one short and one long triangle via a combination of the equations (\ref{eq:thetar_b}) or (\ref{eq:thetal_a}), and (\ref{eq:thetar_c}), or (\ref{eq:thetal_b}). Via equations (\ref{eq:Wmin}) and (\ref{eq:Wmax}) we obtain the values $\omega_-$ and $\omega_+$. Since $\Theta(y)>0$ (see the facts in \ref{rem:data}) Theorem \ref{existvect}, guarantees the existence of a family $\Delta(\mbox{\sc x}, \mbox{\sc y}_\theta)\subset \Gamma(\mbox{\sc x}, \mbox{\sc y}_\theta)$ for $\theta \in [\omega_-,\omega_+)=I(y)$ (or $\theta \in (\omega_-,\omega_+]$). Note that by construction $\Gamma(\mbox{\sc x}, \mbox{\sc y}_{\omega_-})$ or $\Gamma(\mbox{\sc x}, \mbox{\sc y}_{\omega_+})$ admit a {\sc cc} isolated point. 

If $y\in B_3$, the endpoints $\mbox{\sc x},\mbox{\sc y}_\theta \in T\mathbb R^2$ have associated two long triangles, according to equations (\ref{eq:thetar_c}) and (\ref{eq:thetal_b}). Via equations (\ref{eq:Wmin}) and (\ref{eq:Wmax}) we obtain the values $\omega_-$ and $\omega_+$. Since $\Theta(y)>0$, again Theorem \ref{existvect}, guarantees the existence of a family $\Delta(\mbox{\sc x}, \mbox{\sc y}_\theta)\subset \Gamma(\mbox{\sc x}, \mbox{\sc y}_\theta)$ for $\theta \in [\omega_-,\omega_+]=I(y)$. 

If $y\in B_4=B^c$ we have that $\omega_+-\omega_-<0$, implying that there is no bounded isotopy class.

If $x=y$ then by Theorem 3.9 in \cite{paperc} we conclude that $\Gamma(\mbox{\sc x}, \mbox{\sc y}_\theta)$, $\theta\in [-\pi,\pi]$ admits an isolated point, being a path of length zero.  Since $\Gamma(\mbox{\sc x}, \mbox{\sc y}_\theta)$ admits only closed paths, they have parallel tangents, see \cite{papere}. Therefore, these closed paths are bounded-homotopic to paths of arbitrary length, see Proposition 3.8 in \cite{papere}. By Theorem 7.12 in \cite{paperc} none of these paths can be in a bounded isotopy class. Therefore, there is no bounded isotopy class for all $\theta \in (-\pi,\pi)$.
\end{proof}

%DUBINS QUESTION
%Let B be the set of e such that Fe is not arc-wise connected. It seems likely that B is a bounded open set. In particular, if e = (u, a), then, as e ranges over B, we guess that u ranges over a bounded subset of the plane and, what is less intuitive, that a ranges over a bounded set of angles. Moreover, there is a reasonable chance that if e e B, then Fe consists of precisely two components, F(l, e) and F(2, e). 

It is easy to see that there is a natural correspondence between $B\times I(y)$ and $\mathcal B$. Equivalently a correspondence between $B\times I(y)$ and the elements in $\mbox{\sc x},\mbox{\sc y}_\theta \in T\mathbb R^2$ so that there exists a bounded isotopy class. 

%----------------------------------------------------------------------------------------
%	DUBINS QUESTION A
%----------------------------------------------------------------------------------------
\begin{theorem} \label{noopnoclofib} The set of endpoints $\mbox{\sc x},\mbox{\sc y}_\theta \in T\mathbb R^2$ so that $\Gamma(\mbox{\sc x}, \mbox{\sc y}_\theta)$ is a fiber of type:
\begin{itemize}
\item  I, II, or III is bounded, neither open nor closed in $T\mathbb R^2$.
\item  IV is unbounded, neither open nor closed in $T\mathbb R^2$.
\item V is a unit circle.
\end{itemize}
\end{theorem}

\begin{proof} Consider $\mbox{\sc x} \in T\mathbb R^2$ and $y\in B$. Since $B\subset \mathbb R^2$ and $I(y)\subset (-\pi,\pi)$ are both bounded we have that $B\times I(y)$ is bounded. Note that $B$ is not open, since it contains the positive abscissa of the circles (\ref{eq:circ_c}) and  (\ref{eq:circ_d}) i.e., the image of {\sc c} isolated points. The set $B$ is not closed since the point $y=(0,1)$ is in the closure of $B$ but not in$B$, due the existence of parallel tangents, see Proposition 3.8 in \cite{papere}.
 
For the second statement, note that the set of endpoints $\mbox{\sc x},\mbox{\sc y}_\theta \in T\mathbb R^2$ so that $\Gamma(\mbox{\sc x}, \mbox{\sc y}_\theta)$ is a fiber of type $IV$ is unbounded since the fibers of type $IV$ have their final positions in $B_4=B^c$ being this set unbounded. Since $B$ is neither open nor closed, so its complement. Therefore $B_4\times (-\pi,\pi]$ is unbounded neither open nor closed.

Recall that the set of endpoints $(x,X),(y,Y_\theta)\in T\mathbb R^2$ with $x=y$ are such that $\Gamma(\mbox{\sc x}, \mbox{\sc y}_\theta)$ is a fiber of type $V$. Since $(x,X)$ remains fixed while $Y_\theta=e^{\theta i}$, $\theta\in (-\pi,\pi]$ the result follows.
\end{proof}

Next, we establish that $\mathcal B$ is bounded neither open nor closed, answering a question raised by Dubins in pp. 480 in \cite{dubins 2}. 

%----------------------------------------------------------------------------------------
%	DUBINS QUESTION B
%----------------------------------------------------------------------------------------
\begin{corollary} \label{noopnoclo}  $\mathcal B\subset T\mathbb R^2$ is neither open nor closed.
\end{corollary}
\begin{proof}  Immediate from Theorem \ref{noopnoclofib} and the obvious correspondence between $\mathcal B$ and $B\times I(y)$.
\end{proof}

\begin{corollary}\label{cor:param}Consider $\mbox{\sc x}, \mbox{\sc y}_\theta \in T\mathbb R^2$ with $y\in B$. Then:
	\begin{itemize}
	         \item isolated points of zero length are parametrized in the unit circle.
		\item	 isolated points of type {\sc c} are parametrized  in $(0,\pi)\sqcup (0,\pi)$. 
		\item	 isolated points of type {\sc cc} are parametrized  in $$ (0,\pi)\times (0,\pi) \sqcup (0,\pi)\times (0,\pi).$$
		\item The bounded isotopy classes are parametrized in $B\times I(y)$.
		\end{itemize}
\end{corollary}
\begin{proof}  The first and fourth statements were proven in Theorem \ref{noopnoclofib}. The second and third statements are immediate.
\end{proof}

%----------------------------------------------------------------------------------------
%	UPDATE OF THE CLASSIFICATION: SECTION
%----------------------------------------------------------------------------------------
\section{On the classification of the homotopy classes of bounded curvature paths}

Next we present an updated version of Theorem 6.2 in \cite{paperd} by considering the existence of isotopy classes in terms of the values of the class range function.

We first revise Remark \ref{gammaparameter}. Given $\mbox{\sc x}\in T{\mathbb R}^2$,
$$\Gamma=\bigcup_{\substack{{y \in \mathbb R^2}\\ \theta \in (-\pi,\pi]}}\Gamma(\mbox{\sc x}, \mbox{\sc y}_\theta).$$ 

Given $\mbox{\sc x}, \mbox{\sc y} \in T\mathbb R^2$. Let,
$$\Gamma(\mbox{\sc x}, \mbox{\sc y})=\bigcup_{\substack{n \in \mathbb Z}}\Gamma(n)$$ 
where
$$\Gamma(n)=\{\gamma\in \Gamma(\mbox{\sc x,y}): \tau(\gamma)=n, n\in \mathbb Z\},$$ 
with $\tau(\gamma)$ being the turning number\footnote{In \cite{paperd} we used the analogous idea of turning number by considering closed path.} of $\gamma$, see Definition 4.1 in \cite{paperb}. 

For each $\theta \in (-\pi,\pi]$ we have that,
$$\Gamma_n(\mbox{\sc x}, \mbox{\sc y}_\theta)=\{\gamma\in \Gamma(\mbox{\sc x},\mbox{\sc y}_\theta): \tau(\gamma)=n, n\in \mathbb Z\}.$$ 
Suppose that $\Delta(\mbox{\sc x}, \mbox{\sc y}_\theta)\subset \Gamma_k(\mbox{\sc x}, \mbox{\sc y}_\theta)$, for some $\theta\in (-\pi,\pi]$, $k\in \mathbb Z$. The space $\Delta'(\mbox{\sc x}, \mbox{\sc y}_\theta)\subset \Gamma_k(\mbox{\sc x}, \mbox{\sc y}_\theta)$ is the space of paths bounded-homotopic to paths with self-intersections. In \cite{paperd} we proved that:
$$\Delta(\mbox{\sc x}, \mbox{\sc y}_\theta)\cup \Delta'(\mbox{\sc x}, \mbox{\sc y}_\theta)=\Gamma_k(\mbox{\sc x}, \mbox{\sc y}_\theta).$$

The proof of Theorem \ref{paramclass} is immediate from the facts in \ref{rem:data}, Theorem \ref{noopnoclofib} and Theorem 6.2 in \cite{paperd}.

%----------------------------------------------------------------------------------------
%	THEOREM: UPDATED CLASSIFICATION
%----------------------------------------------------------------------------------------
\begin{theorem}\label{paramclass} Choose $\mbox{\sc x}, \mbox{\sc y}_{\theta} \in T\mathbb R^2$ we have that:
\begin{equation} 
	\label{eq:main1}
		 \Gamma(\mbox{\sc x}, \mbox{\sc y}_\theta)=\bigcup_{\substack{n \in \mathbb Z}}\Gamma_n(\mbox{\sc x}, \mbox{\sc y}_\theta),\hspace{.2cm} \theta\in(-\pi,\pi].
\end{equation} 

\begin{enumerate}
\item If $\Theta(y)>0$ there exists a family of bounded isotopy classes $\Delta(\mbox{\sc x}, \mbox{\sc y}_\theta)$ so that:

\begin{equation} 
	\label{eq:main2}
		 \Gamma_k(\mbox{\sc x}, \mbox{\sc y}_\theta)=\Delta(\mbox{\sc x}, \mbox{\sc y}_\theta) \cup \Delta'(\mbox{\sc x}, \mbox{\sc y}_\theta), \hspace{.2cm}  \mbox{for}\hspace{.2cm}  \theta\in I(y), \hspace{.2cm}  \mbox{and some}\hspace{.2cm}  k\in \mathbb Z.
	\end{equation} 
	In particular, if: 
	\begin{itemize}
	
\item	If $y\in B_1$ then $\Gamma(\mbox{\sc x}, \mbox{\sc y}_\theta)$ is a fiber to type I. 
\item If $y\in B_2$ then $\Gamma(\mbox{\sc x}, \mbox{\sc y}_\theta)$ is a fiber to type II.
\item  If $y\in B_3$ then $\Gamma(\mbox{\sc x}, \mbox{\sc y}_\theta)$ is a fiber to type III.
\item  If $y\in B_4$ then $\Gamma(\mbox{\sc x}, \mbox{\sc y}_\theta)$ is a fiber to type IV.
\item  If $y=x$ then $\Gamma(\mbox{\sc x}, \mbox{\sc y}_\theta)$ is a fiber to type V.
\end{itemize}

\item If $y\in B$, $y\neq x$ and $\Theta(y)=0$ we may have a {\sc c} or a {\sc cc} isolated point. 
%And, there is no (non-empty interior) bounded isotopy class.

	\item	If $\omega_+(y)-\omega_-(y)<0$ we have that there is no bounded isotopy class.	
\end{enumerate}
\end{theorem}

%----------------------------------------------------------------------------------------
%	APPENDIX
%----------------------------------------------------------------------------------------
\begin{center} {\sc Appendix\\ Homotopy classes and deformations of Dubins paths}
\end{center}

We would like to motivate a theory analyzing algorithmic aspects of deformations of piecewise bounded curvature paths of constant curvature. Many standard questions in computational geometry can be adapted for this class of paths. 

In \cite{paperd} we defined operations on bounded curvature paths being a finite number of concatenations of line segments and arcs of unit radius circles, the so-called $cs$ paths. The line segments and arcs of circles are called components. The number of components is called the complexity of the path\footnote{Dubins paths have complexity at most 3.}. Also in \cite{paperd}, we proved that a $cs$ path can be constructed arbitrarily close to any given bounded curvature path. It is of interest to study the computational complexity of deforming $cs$ paths.

It is not hard to see that for any given $\mbox{\sc x,y}\in T\mathbb R^2$ the space $\Gamma(\mbox{\sc x,y})$ has a finite number of Dubins paths. In Example \ref{ex:spaces} we index Dubins paths according to their length. The length minimizer in $\Gamma(\mbox{\sc x,y})$ is denoted by $\gamma_0$.  

Next, we relate the types of connected components, the number of local minima, number of global minima, existence of local maxima, and deformations of $cs$ paths. In Fig. \ref{fig:spaces} we consider seven illustrations, and in Example \ref{ex:spaces} we consider seven items. We associate illustrations and items in an obvious way. 

%----------------------------------------------------------------------------------------
%	FIGURE 14
%----------------------------------------------------------------------------------------
\begin{figure}[h]
	\centering
	\includegraphics[width=1\textwidth,angle=0]{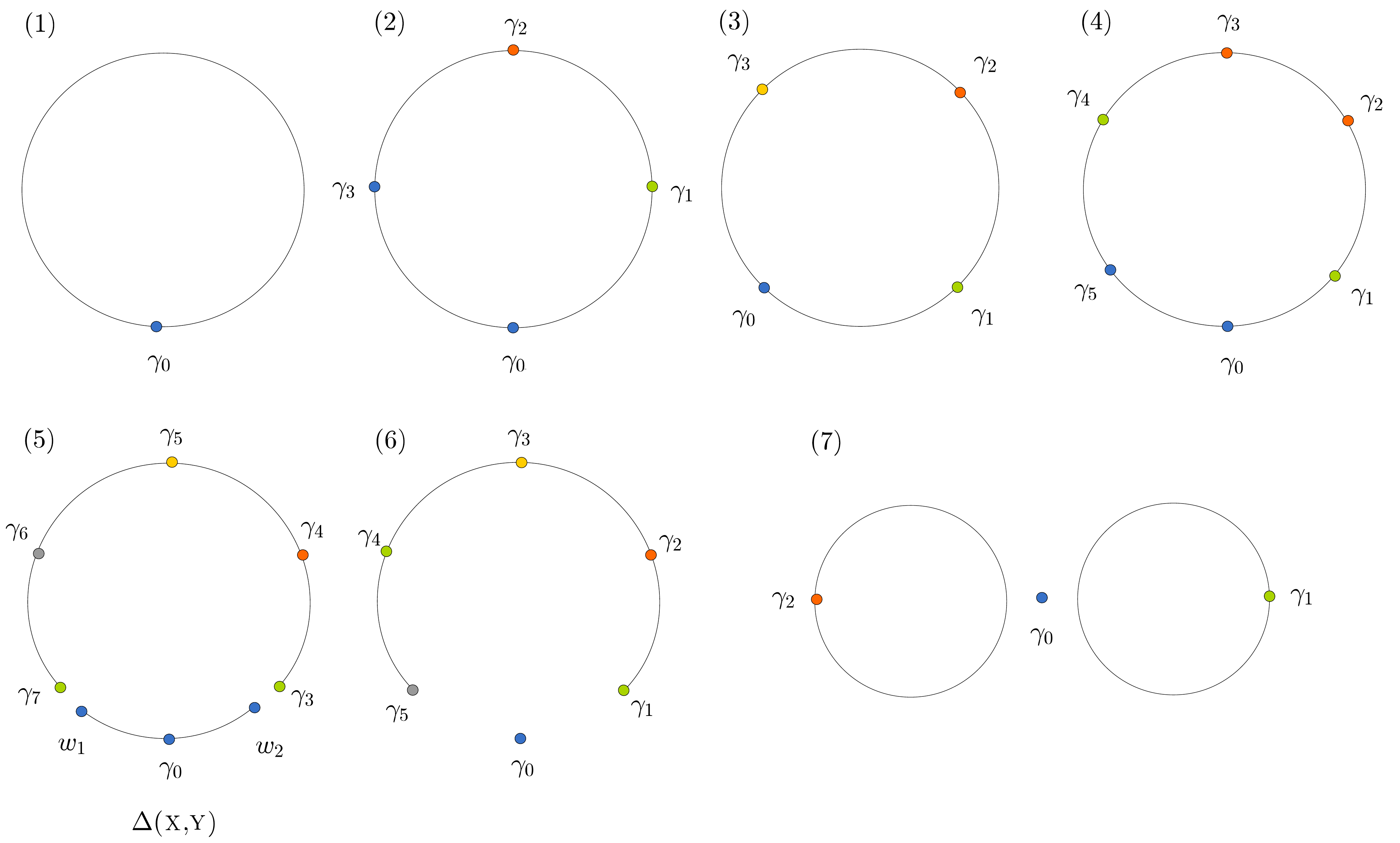}
	\caption{A schematic representation for spaces of bounded curvature paths. Each of the seven illustrations represent the connected components $\Gamma(n)\subset \Gamma(\mbox{\sc x,y})$ admitting Dubins paths, here represented by points. Points with the same color suggest that the associated paths are bounded-homotopic, see Figs. \ref{sixdubb1}, \ref{fourdub}-\ref{sixdubb3}.}
	\label{fig:spaces}
\end{figure}
\begin{example}\label{ex:spaces}Consider $x=(0,0)$, $X=e^{2\pi i}\in T_x\mathbb R^2$. In Figure \ref{fig:spaces} we illustrate spaces $\Gamma(\mbox{\sc x,y})$ such that: \hfill
\begin{enumerate} 

\item $y=(z,0)$, $z\in\mathbb R^+$, $Y=e^{2\pi i}\in T_y\mathbb R^2$. This example corresponds to the Euclidean geometry case (up to isometries) where the single length minimizer between any two points is a line segment. 

\item  $y=(4,-8)$, $Y=e^{2\pi i}\in T_y\mathbb R^2$. There are four Dubins paths, $\gamma_0$ being the length minimizer, see Fig. \ref{fourdub}. The paths $\gamma_0$ and $\gamma_3$ are bounded-homotopic. This is checked in Proposition 4.3 and Fig. 13 in \cite{paperd}.

\item $y=(z,0)$, $z\geq 4$, $Y=e^{\pi i}\in T_y\mathbb R^2$. There are four Dubins paths, with $\gamma_0$ and $\gamma_1$ being length minimizers. These four paths are not bounded-homotopic one to the other. In Fig. 1 in \cite{paperb} we illustrate the two length minimizers.  

\item $y=(-2,1)$, $Y=e^{-\frac{\pi}{4} i}\in T_y\mathbb R^2$. There are six Dubins paths, one being the length minimizer, see Fig. \ref{sixdubb1}. The paths $\gamma_0$  and $\gamma_5$; $\gamma_1$ and $\gamma_4$, and $\gamma_2$ and $\gamma_3$ are pair-wise bounded-homotopic. This can be verified by applying Proposition 4.4 in \cite{paperd}.

 \item $y=(3,0)$, $Y=e^{\frac{\pi}{3} i}\in T_y\mathbb R^2$. Since $y\in B$, then $\Theta(y)>0$, so we have that there exists a bounded isotopy class $\Delta(\mbox{\sc x,y})$, or equivalently $\Omega\neq \emptyset$. In this case, the length minimiser in $\Gamma(\mbox{\sc x,y})$ is a unique {\sc csc} path and it is an element in $\Delta(\mbox{\sc x,y})$. This is a consequence of Proposition 2.13 in \cite{papera} and Theorem 8.1 in \cite{paperc}. 

 Note that there are eight Dubins paths, one being the length minimizer ($\gamma_0$ lies in $\Omega$), see Fig. \ref{sixdubb2}. In addition, the paths $\gamma_0$, $w_1$ and $w_2$ are bounded-isotopic one to the other since they are paths in $\Delta(\mbox{\sc x,y})$, see Theorem 5.4 in \cite{paperd}. It seems plausible to think that $w_1$ and $w_2$ are local maxima (not local minima) of length. In addition, the path $\gamma_3$ is the length minimizer in $\Delta'(\mbox{\sc x,y})$.

\item there are six Dubins paths, the length minimizer being an isolated point, see Fig. \ref{sixdubb3} and Theorem 3.9 in \cite{paperc}. By a similar argument as the one in Proposition 4.4 in \cite{paperd} we conclude that $\gamma_1$ is bounded-homotopic to $\gamma_4$.

\item  $x=y$ and $X=Y$. Closed bounded curvature paths are not bounded-homotopic to a single point. In this case, the length minimiser $\gamma_0$ is an isolated point of length zero, see Theorem 3.9 in \cite{paperc}.  There are two non-trivial length minimisers, say $\gamma_1$ and $\gamma_2$. These paths lie in the adjacent circles $C_l(\mbox{\sc x})$ and $C_r(\mbox{\sc x})$ respectively. It is easy to see that $\gamma_1$ and $\gamma_2$ lie in different homotopy classes since they have winding number $1$ and $-1$ respectively, see Theorem 4.6 in \cite{paperb}.
\end{enumerate}
\end{example}

After the previous examples, a natural task would be to determine for any pair of endpoints $\mbox{\sc x,y}\in T\mathbb R^2$ the exact number of homotopy classes admitting Dubins paths. This should be done after first describing all the possible scenarios for homotopies between Dubins paths. 

A closely related problem is the following. Given $\mbox{\sc x},\mbox{\sc y}_\theta\in T\mathbb R^2$, $\theta \in (-\pi,\pi]$. Describe how the number of Dubins paths vary, as we vary $\theta \in (-\pi,\pi]$. Also, describe how the type of all (up to eight?) Dubins paths vary for all the fibers. A description of the type of the global minimum (first Dubins path) has been obtained in \cite{bui}.

Given $ \mbox{\sc x},\mbox{\sc y}\in T\mathbb R^2$. What are the complexity $n>3$ $cs$ paths of minimal length? For certain pairs the answer is trivial. What if $d(x,y)<4$?

Given two $cs$ paths with prescribed complexity and lying in the same homotopy class. What is the minimal number of operations (or moves) to deform one path into the other? What are these moves?

%----------------------------------------------------------------------------------------
%	FIGURE 15
%----------------------------------------------------------------------------------------
\begin{figure}[h]
	\centering
	\includegraphics[width=.9\textwidth,angle=0]{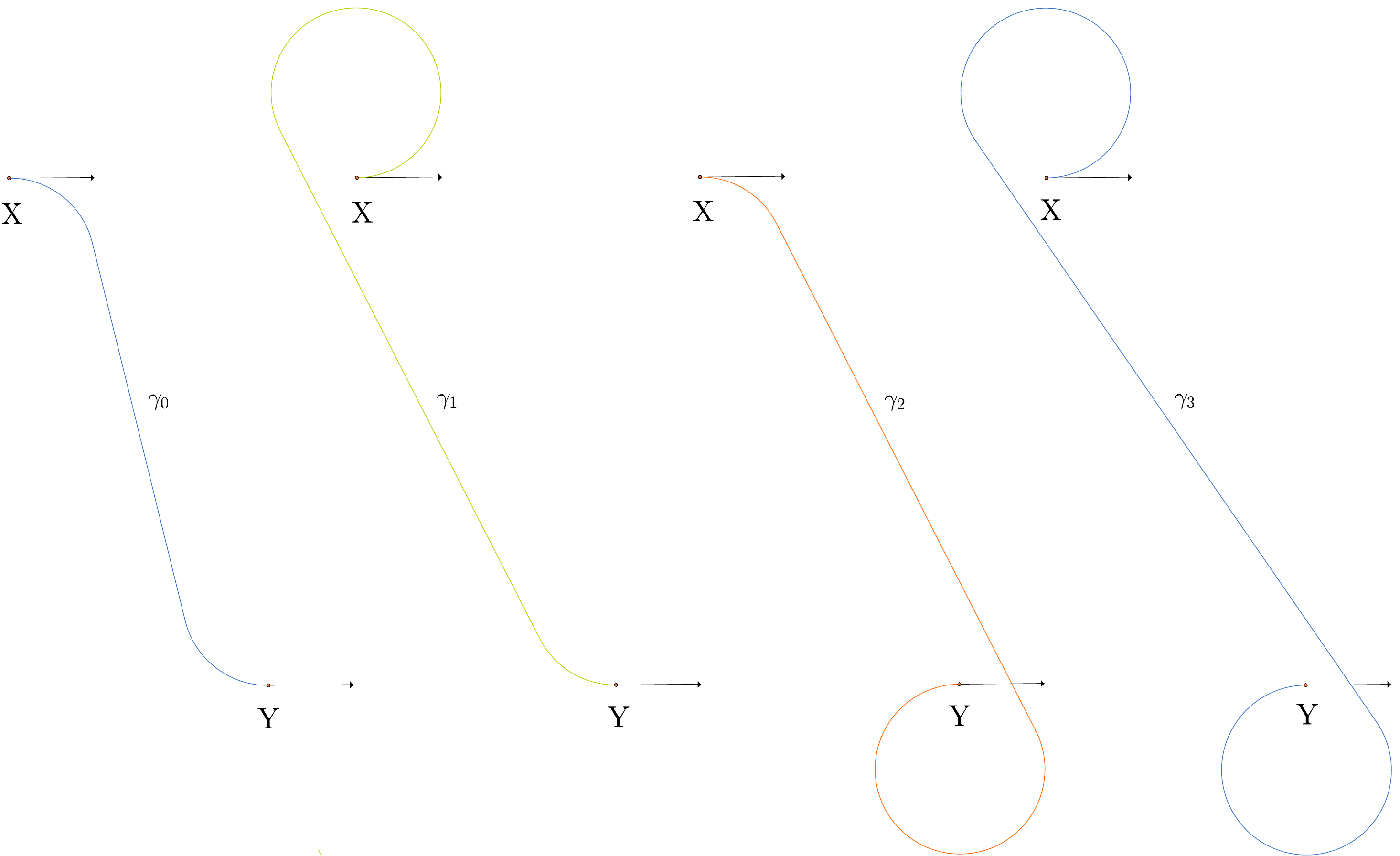}
	\caption{Spaces of bounded curvature paths may have four local minima of length. Note that $\gamma_0$ and $\gamma_3$ are bounded-homotopic, see Proposition 4.3 in \cite{paperd}.}
	\label{fourdub}
\end{figure}

%----------------------------------------------------------------------------------------
%	FIGURE 16
%----------------------------------------------------------------------------------------
\begin{figure}[h]
	\centering
	\includegraphics[width=1\textwidth,angle=0]{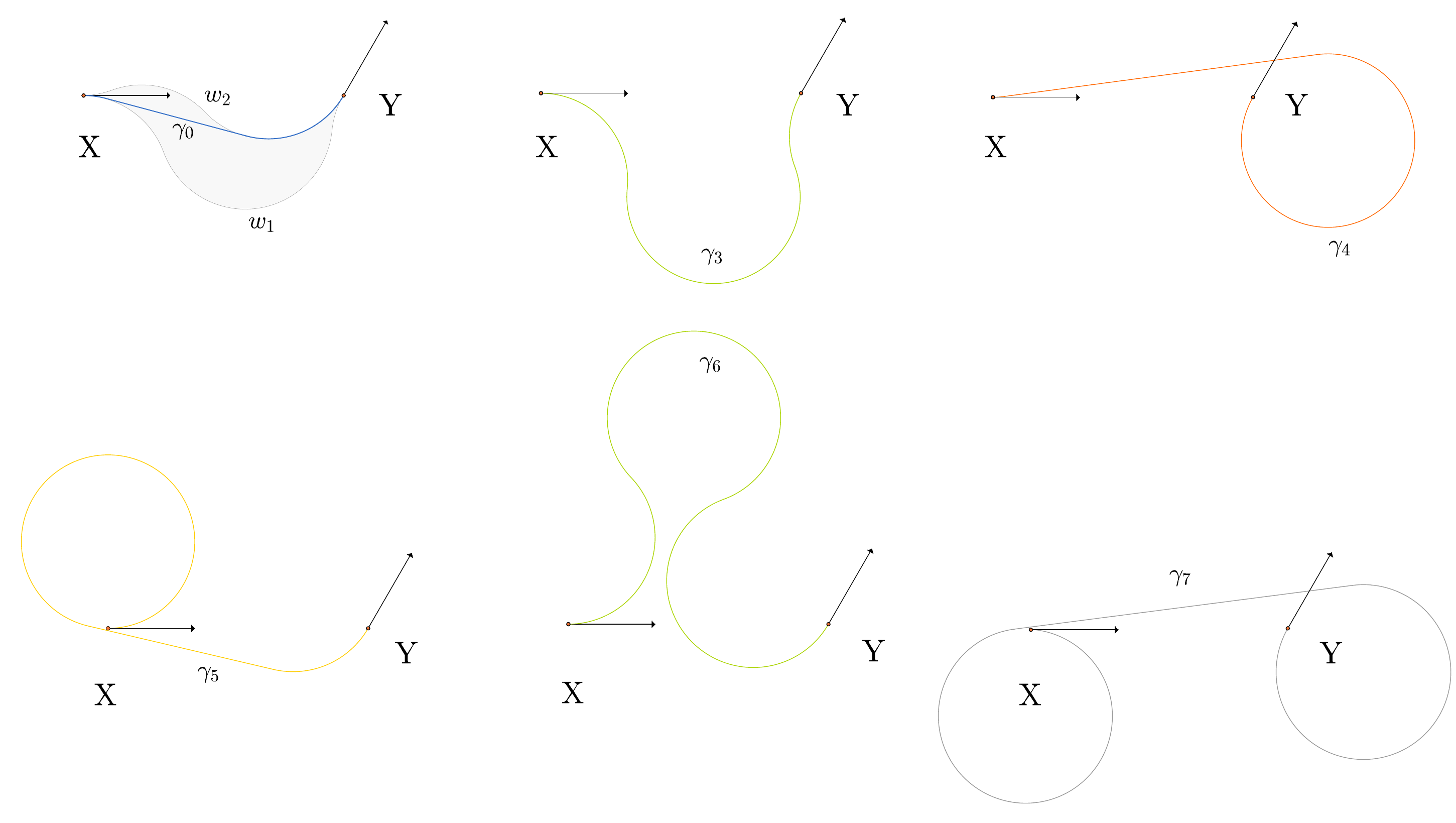}
	\caption{Spaces of bounded curvature paths may have up to eight {\sc csc-ccc} paths, local minima (or maxima) of length. Note that $\gamma_3$ and $\gamma_6$ are bounded-homotopic.}
	\label{sixdubb2}
\end{figure}

%----------------------------------------------------------------------------------------
%	FIGURE 17
%----------------------------------------------------------------------------------------
\begin{figure}[h]
	\centering
	\includegraphics[width=.9\textwidth,angle=0]{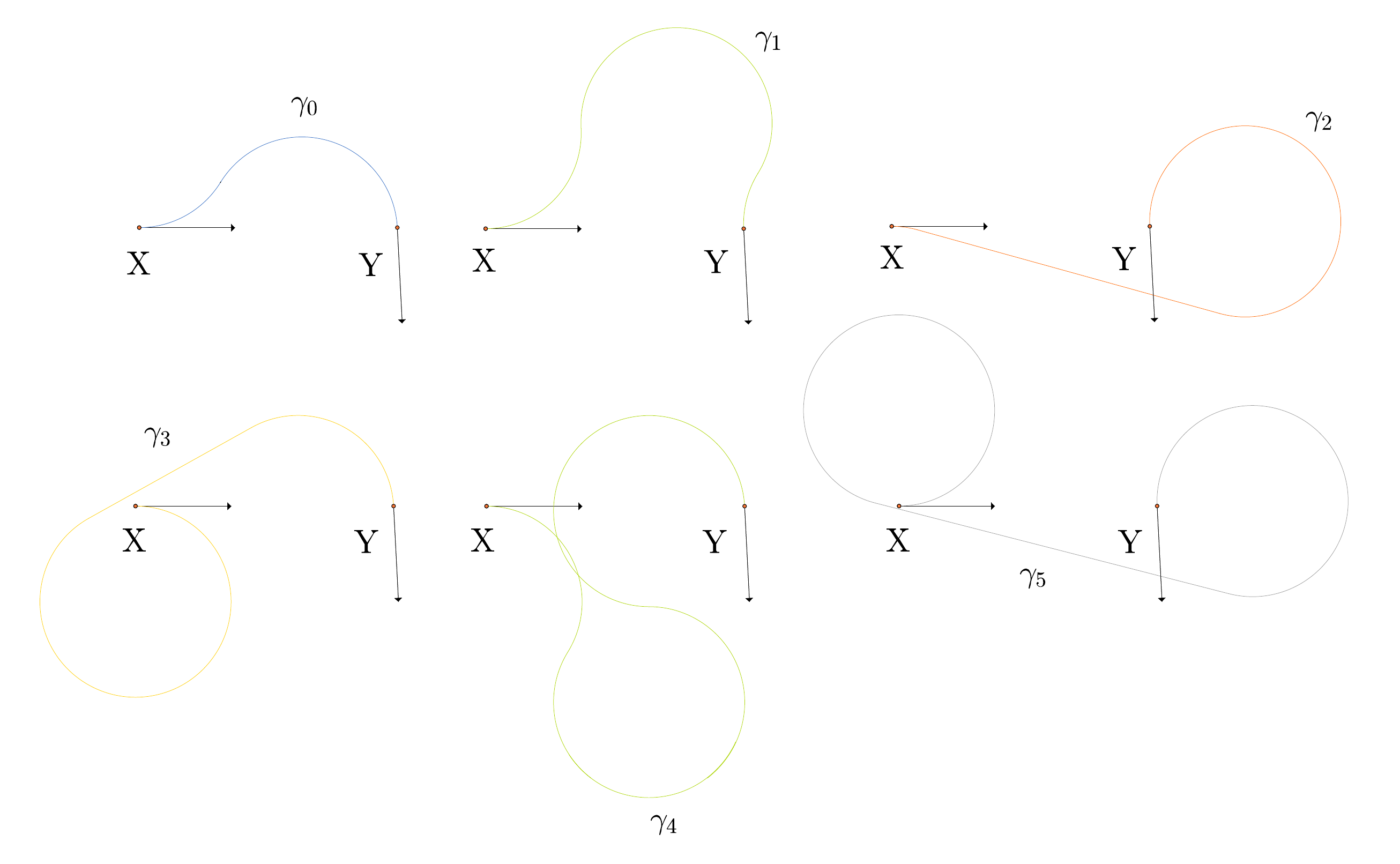}
	\caption{Spaces of bounded curvature paths may have six local minima of length. Note that $\gamma_1$ and $\gamma_4$ are bounded-homotopic. It is of interest to classify the fibers $\Gamma(\mbox{\sc x},\mbox{\sc y}_\theta)$ in terms of the way the type and number of Dubins paths changes as $\theta$ varies, see also Figs. \ref{sixdubb1}, \ref{fourdub} and \ref{sixdubb2}. }
	\label{sixdubb3}
\end{figure}

%----------------------------------------------------------------------------------------
%	REFERENCES
%----------------------------------------------------------------------------------------
\bibliographystyle{amsplain}
   
\end{document}